\documentclass[a4paper]{article}

\usepackage{a4wide}
\usepackage{amsmath}
\usepackage{amssymb} 
\usepackage{bm}
\usepackage{booktabs}
\usepackage{algorithm}
\usepackage{multicol}
\usepackage{multirow}
\usepackage[T1]{fontenc}
\usepackage{lmodern, textcomp}
\usepackage{amsfonts}
\usepackage{graphicx}
\usepackage{epstopdf}
\usepackage{algorithmic}
\usepackage{cleveref}
\usepackage{amsthm}
\usepackage{url}
\usepackage{color}

\newtheorem{proposition}{Proposition}[section]
\newtheorem{theorem}{Theorem}[section]
\newtheorem{corollary}[theorem]{Corollary}

\title{Cross-interactive residual smoothing for global and block Lanczos-type solvers for linear systems with multiple right-hand sides\thanks{\textbf{Funding:} This work was supported by grant numbers JP16K17639, JP17K12690, JP18H03250, JP18K18064, JP19KK0255, JP20K14356, JP21H03451, and JP21K11925 from the Grants-in-Aid for Scientific Research Program (KAKENHI) of the Japan Society for the Promotion of Science (JSPS).}}

\date{}

\author{Kensuke Aihara\thanks{Department of Computer Science, Faculty of Information Technology, Tokyo City University, 1-28-1 Tamazutsumi, Setagaya-ku, Tokyo 158-8557, Japan 
  ({\tt aiharak@tcu.ac.jp}).}
\and Akira Imakura\thanks{Faculty of Engineering, Information and Systems, University of Tsukuba, 1-1-1 Tennodai, Tsukuba, Ibaraki 305-8573, Japan 
  ({\tt imakura@cs.tsukuba.ac.jp}, {\tt morikuni@cs.tsukuba.ac.jp}).}
\and Keiichi Morikuni\footnotemark[3]}

\begin{document}

\maketitle

\begin{abstract}
Global and block Krylov subspace methods are efficient iterative solvers for large sparse linear systems with multiple right-hand sides. 
However, global or block Lanczos-type solvers often exhibit large oscillations in the residual norms and may have a large residual gap relating to the loss of attainable accuracy of the approximations. 
Conventional residual smoothing schemes suppress these oscillations but cannot improve the attainable accuracy, whereas a recent residual smoothing scheme enables the improvement of the attainable accuracy for single right-hand side Lanczos-type solvers. 
The underlying concept of this scheme is that the primary and smoothed sequences of the approximations and residuals influence one another, thereby avoiding the severe propagation of rounding errors. 
In the present study, we extend this cross-interactive residual smoothing to the case of solving linear systems with multiple right-hand sides. 
The resulting smoothed methods can reduce the residual gap with a low additional cost compared to their original counterparts. 
We demonstrate the effectiveness of the proposed approach through rounding error analysis and numerical experiments.
\end{abstract}

\


\


\section{Introduction}

We consider linear systems with multiple right-hand sides 
\begin{align}
AX=B, \label{AX=B}
\end{align}
where $A \in \mathbb{R}^{n\times n}$ is a large sparse nonsymmetric and nonsingular matrix and $B := [\bm{b}_1, \bm{b}_2, \dots, \bm{b}_s] \in \mathbb{R}^{n\times s}$ is a rectangular matrix with $s \ll n$. 
This problem appears in various fields of scientific computing (e.g., see \cite{ElGuennouni_2003,Nakamura_2012,Sakurai_2003,Sakurai_2010,Zhang_2010} and their references), and several types of iterative solvers have been studied \cite[section~12.4]{meurant2020krylov}. 
Global Krylov subspace methods such as the global bi-conjugate gradient stabilized method (Gl-BiCGSTAB) \cite{Jbilou_2005} and generalized global conjugate gradient squared methods including Gl-CGS2 \cite{Zhang_2010} have been developed. 
Global methods generate the approximations using the matrix Krylov subspace $\mathcal K_k^G(A, R_0) := \big\{ \sum_{i=0}^{k-1} c_i A^iR_0 \mid c_i \in \mathbb{R}\big\}$, where $R_0 := B-AX_0 \in \mathbb{R}^{n\times s}$ is the initial residual with an initial guess $X_0$. 
Moreover, global methods correspond to standard Krylov subspace methods applied to a linear system $(I_{sn} \otimes A) \bm{x} = \bm{b}$, where $\otimes$ denotes the Kronecker product, $\bm{b} := [\bm{b}_1^\top,\bm{b}_2^\top,\dots,\bm{b}_s^\top]^\top$, and $X := [\bm{x}_1,\bm{x}_2,\dots,\bm{x}_s]$ for $\bm{x} := [\bm{x}_1^\top,\bm{x}_2^\top,\dots,\bm{x}_s^\top]^\top$. 
Thus, many results of Krylov subspace methods for a single linear system can be naturally extended to the case of solving \eqref{AX=B}. 
Furthermore, global methods can easily be applied to general linear matrix equations such as the Sylvester equation; for example, see \cite{Beik_2011,Jbilou_1999}. 
Another approach is the use of block Krylov subspace methods \cite{Gutknecht_2006}, such as the block BiCGSTAB method (Bl-BiCGSTAB) \cite{ElGuennouni_2003}, which uses the block Krylov subspace $\mathcal K_k^B(A, R_0) :=\big\{ \sum_{i=0}^{k-1} A^iR_0 \gamma_i \mid \gamma_i \in \mathbb{R}^{s\times s}\big\}$. 
As the search subspace for a column of the approximation $X_k$ expands with $s$ dimensions at each iteration, block methods can achieve faster convergence than their single counterparts \cite{OLeary_1980}. 
Block methods may be numerically unstable for a larger $s$; thus, stabilization strategies that orthonormalize the iteration matrices have also been developed; for example, see \cite{Nakamura_2012,OLeary_1980}.

Global Lanczos-type solvers (methods based on the global Lanczos process \cite{Jbilou_2005}), such as Gl-BiCGSTAB and Gl-CGS2, are natural extensions of standard Lanczos-type solvers, such as BiCGSTAB \cite{van_der_Vorst_1992} and CGS2 \cite{Fokkema_1996}, respectively. 
Solvers of this type often use the following recursion formulas to update the approximation and the corresponding residual:
\begin{align}
X_{k+1} = X_k + \alpha_k P_k,\quad R_{k+1} = R_k - \alpha_k (AP_k), \label{X_R}
\end{align}
where $P_k \in \mathbb{R}^{n\times s}$ is a direction matrix and $\alpha_k \in \mathbb{R}$ is a scalar coefficient. Note that the concrete choices of $P_k$ and $\alpha_k$ are determined by specific solvers. 
Starting with the initial residual $R_0 := B - AX_0$, the equality $R_k = B - AX_k$ holds for $k=1,2,\dots$ in exact arithmetic, but this equality may not hold in finite precision arithmetic owing to the accumulation of rounding errors in $X_k$ and $R_k$. 
This difference between the recursively updated residual $R_k$ and explicitly computed residual $B-AX_k$ is referred to as the residual gap. 
As in the case of solving a single linear system, global (and block) Lanczos-type solvers often suffer from a large residual gap as a result of rounding errors. 
The residual gap $G_{R_k} := (B - AX_k) - R_k$ is important because the explicitly computed residual norm $\|B-AX_k \|$ (referred to as the true residual norm) is bounded as follows:
\begin{align*}
\|G_{R_k}\| - \|R_k\| \leq \|B-AX_k \| \leq \|G_{R_k}\| + \|R_k\|, 
\end{align*}
where $\|\cdot \|$ denotes the Frobenius norm. 
Therefore, when $\|R_k\|$ becomes sufficiently small, the attainable accuracy of $X_k$ in terms of the true residual norm is dependent on $\|G_{R_k} \|$.

The residual gap that is observed in the standard Lanczos-type solvers for a single linear system has been thoroughly analyzed and remedies have been proposed. 
In particular, if the maximum of the residual norms is relatively large compared to the initial residual norm, a large residual gap and a loss of attainable accuracy will appear, and refined techniques to avoid a large increase in the residual norms have been studied (see, for example, \cite{Aihara_2019,Sleijpen_1996}). 
As we demonstrate later, a large relative residual norm leads to a large residual gap in \eqref{X_R}. 
Therefore, in the current study, we focus on a recent residual smoothing scheme \cite{Aihara_2019} to reduce the residual gap.

Residual smoothing is a well-known technique for generating a residual sequence that decreases smoothly. 
The original smoothing scheme \cite{Schonauer_1987,Weiss_1996} is a simple transformation that provides smoothed residuals by using a linear combination of the primary (non-smoothed) residuals obtained from an iterative method. 
Zhou and Walker \cite{Zhou_1994} suggested an alternative smoothing scheme that does not directly require the primary residuals and approximations. 
However, it has been shown in \cite{Gutknecht_2001} that these conventional smoothing schemes do not aid in improving the attainable accuracy; the true residual norms of the smoothed sequences stagnate at the same order of magnitude as the primary ones. 
The root cause of this phenomenon is that rounding errors accumulated in the primary sequences propagate directly to the smoothed sequences, because the smoothed sequences are computed from the primary sequences one-sidedly. 
In contrast, Zhou--Walker's scheme was recently modified in \cite{Aihara_2019} so that the primary and smoothed sequences influence one another and severe propagation of rounding errors can be avoided. 
We refer to this scheme as {\it cross-interactive residual smoothing (CIRS)} and extend it to be applicable for global and block Lanczos-type solvers. 
Following \cite{Aihara_2019}, we present rounding error analysis and numerical experiments to show that CIRS is effective in reducing the residual gap when solving systems with multiple right-hand sides.

We note that the recursion formulas used in block Lanczos-type solvers, such as Bl-BiCG \cite{OLeary_1980} and Bl-BiCGSTAB, are partially different from \eqref{X_R}; the following forms are often used: 
\begin{align}
X_{k+1} = X_k + P_k\alpha_k^\square, \quad R_{k+1} = R_k - (AP_k)\alpha_k^\square, \label{B_X_R}
\end{align}
where $\alpha_k^\square \in \mathbb{R}^{s\times s}$ is utilized instead of a scalar coefficient $\alpha_k$, and it is determined under various conditions (e.g., an orthogonality or minimization condition) depending on the solvers. 
For an example, see \cite{ElGuennouni_2003}. 
In this case, a large residual gap occurs not only because of the large relative residual norm \cite{Tadano_2009}. 
However, as we will show later, when using the stabilization strategy---that is, the orthonormalization of the columns of $P_k$---we can show that the large relative residual norm is a significant factor for the large residual gap.

In contrast, the block BiCGGR method (Bl-BiCGGR) was proposed in \cite{Tadano_2009} as a variant of Bl-BiCGSTAB. 
The basic concept of Bl-BiCGGR is to use an alternative recursion formula $R_{k+1} = R_k - A(P_k\alpha_k^\square)$ instead of the second formula in \eqref{B_X_R} to update the residual. 
Numerical experiments showed that Bl-BiCGGR has a smaller residual gap than Bl-BiCGSTAB. 
However, the alternative recursion is essentially the same as \eqref{X_R}; that is, the residual gap may significantly increase in the presence of a large relative residual norm.

The stabilization strategy used in \cite{Nakamura_2012} can also be incorporated with Bl-BiCGGR using a slightly different approach; the columns of $R_k$ are orthonormalized instead of those of $P_k$ \cite{Tadano_2017}. 
In this paper, we demonstrate that CIRS can be applied to both Bl-BiCGSTAB and Bl-BiCGGR as well as these algorithms in combination with stabilization strategies, and that the resulting smoothed methods generate further accurate approximations.

The remainder of this paper is organized as follows. In \cref{sec2}, we discuss the residual gap when using \eqref{X_R} and \eqref{B_X_R}. 
In \cref{sec3}, we present simple and cross-interactive schemes for residual smoothing. 
In \cref{sec4}, rounding error analysis shows that CIRS is effective in reducing the residual gap. 
In \cref{sec5}, we propose several smoothed algorithms for specific global and block Lanczos-type solvers. 
In \cref{sec6}, we describe the application of global algorithms with CIRS to the Sylvester equation. 
In \cref{sec7}, numerical experiments demonstrate the effectiveness of the proposed methods. 
Finally, concluding remarks are presented in \cref{sec8}.

Throughout, we use the inner product $\langle X, Y \rangle_F := \text{tr}(X^\top Y)$ for matrices $X, Y \in \mathbb{R}^{n\times s}$ and $\|\cdot \|$ denotes the associated Frobenius norm; that is, $\| X \| := \sqrt{\langle X, X\rangle_F}$. 
For ease of discussion, we also assume that $\| B \| = 1$ and $X_0 = O$; that is, $R_0 = B$.

\section{Influence of rounding errors in recursion formulas}\label{sec2}

Based on \cite{Aihara_2019,Greenbaum_1997,Sleijpen_1996}, we present a rounding error analysis to demonstrate that a large residual gap may occur when the maximum of the recursively updated residual norms in \eqref{X_R} is relatively large.

For matrix operations in finite precision arithmetic, we use the following models \cite{Higham_2002}: 
\begin{align*}
&\text{fl}(P+Q) =  P + Q + E',\quad \|E'\| \leq {\bf u} (\|P\| + \|Q\|),\\
&\text{fl}(\alpha P) = \alpha P + E'',\quad \|E''\| \leq {\bf u} \|\alpha P\|,\\
&\text{fl}(P\alpha^\square) = P \alpha^\square + E^\square,\quad \|E^\square\| \leq s {\bf u} \|P\| \|\alpha^\square\|,\\
&\text{fl}(AP) = AP + E''',\quad \|E'''\| \leq m {\bf u} \| A \| \| P \|, 
\end{align*}
for given $P, Q \in \mathbb{F}^{n\times s}$, $\alpha \in \mathbb{F}$, and $\alpha^\square \in \mathbb{F}^{s\times s}$. 
Here, $\mathbb{F} \subset \mathbb{R}$ is a set of floating point numbers, fl($\cdot $) denotes the result of floating point computations, ${\bf u}$ is the unit roundoff, and $m$ is the maximum number of nonzero entries per row of $A$. 
We omit terms of $O({\bf u}^2)$ and regard ${\bf u} \| \text{fl}( \ast ) \|$ as ${\bf u} \| \ast \|$. 
Note that, unlike the model of matrix--vector multiplication used in \cite[section~3]{Aihara_2019}, no matrix $\Delta$ exists such that $\text{fl}(AP) = AP + \Delta P$ and $\|\Delta\| \leq m {\bf u} \| A \|$ holds (cf.~\cite[section~3.5]{Higham_2002}).

Similarly to \cite[sections~3.1 and 3.2]{Aihara_2019}, we can evaluate the upper bound of the norm of the residual gap when using \eqref{X_R}; see also \cite[section~2]{Greenbaum_1997} and \cite[section~3.1]{Sleijpen_1996}.

\begin{theorem}\label{Theorem1}
Let $X_k \in \mathbb{F}^{n\times s}$ and $R_k \in \mathbb{F}^{n\times s}$ be the $k$th approximation and residual, respectively, generated by \eqref{X_R} in finite precision arithmetic. 
Then, the norm of the residual gap $G_{R_k} := (B - AX_k) - R_k$ is bounded as follows: 
\begin{align}
\|G_{R_k}\| \leq k(5+2m){\bf u}\|A\| \max_{0< j\leq k} \|X_j\| + 5 (k+1) {\bf u}\max_{0\leq j \leq k} \|R_j\|. \label{gap1} 
\end{align}
\end{theorem}
\begin{proof}
According to the above matrix operation models, the local errors in the updated approximation $X_k= \text{fl}(X_{k-1} + \text{fl}(\alpha_{k-1}P_{k-1}))$ and residual $R_k= \text{fl}(R_{k-1} - \text{fl}(\alpha_{k-1}\text{fl}(AP_{k-1})))$ can be evaluated as follows: 
\begin{align*}
X_k 
&= X_{k-1} + \text{fl}(\alpha_{k-1}P_{k-1}) + E_{X_k}',\quad \|E_{X_k}'\| \leq {\bf u} (2\|X_{k-1}\| + \|X_k\|) \\
&= X_{k-1} + \alpha_{k-1}P_{k-1} + E_{X_k}' + E_{X_k}'',\quad \|E_{X_k}''\| \leq {\bf u} (\|X_{k-1}\| + \|X_k\|) \\
&= X_{k-1} + \alpha_{k-1}P_{k-1} + E_{X_k},\quad \|E_{X_k}\| \leq {\bf u} (3\|X_{k-1}\| + 2\|X_k\|), \\
R_k 
&= R_{k-1} - \text{fl}(\alpha_{k-1}\text{fl}(AP_{k-1})) - E_{R_k}',\quad \|E_{R_k}'\| \leq {\bf u} (2\|R_{k-1}\| + \|R_k\|) \\
&= R_{k-1} - \alpha_{k-1}\text{fl}(AP_{k-1}) - E_{R_k}' - E_{R_k}'',\quad \|E_{R_k}''\| \leq {\bf u} (\|R_{k-1}\| + \|R_k\|) \\
&= R_{k-1} - \alpha_{k-1}AP_{k-1} - E_{R_k}' - E_{R_k}'' - \alpha_{k-1}E_{R_k}''',\\
&\hphantom{=}\quad \|\alpha_{k-1}E_{R_k}'''\| \leq m{\bf u} \|A\|(\|X_{k-1}\| + \|X_k\|) \\
&= R_{k-1} - \alpha_{k-1}AP_{k-1} - E_{R_k},\\
&\hphantom{=}\quad \|E_{R_k}\| \leq {\bf u} (3\|R_{k-1}\| + 2\|R_k\|) + m{\bf u} \|A\|(\|X_{k-1}\| + \|X_k\|), 
\end{align*}
where $E_{X_k} := E_{X_k}'+E_{X_k}''$ and $E_{R_k}:=E_{R_k}' + E_{R_k}'' + \alpha_{k-1}E_{R_k}'''$. 
Thus, the norm of the residual gap can be bounded as follows: 
\begin{align}
\begin{split}
&\|(B-AX_{k}) - R_k\|\\
&\quad= \|B - A(X_{k-1} + \alpha_{k-1}P_{k-1} + E_{X_k})- (R_{k-1} - \alpha_{k-1}AP_{k-1} - E_{R_k})\|\\
&\quad= \|(B-AX_{k-1}) - R_{k-1} - AE_{X_k} + E_{R_k}\|\\
&\quad = \Bigg\| (B-AX_0) - R_0 - A\sum_{j=1}^kE_{X_j} + \sum_{j=1}^kE_{R_j}\Bigg\|\\
&\quad \leq \|A\|\sum_{j=1}^k\|E_{X_j}\| + \sum_{j=1}^k\|E_{R_j}\| \\
&\quad \leq (5+2m){\bf u}\|A\|\sum_{j=1}^k\|X_j\| + 5{\bf u}\sum_{j=0}^k\|R_j\|, 
\end{split}\label{gap_Th1}
\end{align}
where $(B-AX_0) - R_0 = O$ holds under the assumption that $X_0 = O$. 
The proof is completed by bounding the approximation and residual norms by their maximums.
\end{proof}

\Cref{Theorem1} implies that it is important to reduce $\max_j \|X_j\|$ and $\max_j \|R_j\|$ to be as small as possible during the iterations to avoid a large residual gap. 
In our experience, ${\bf u}(\max_j \|R_j\|)$ provides a practical estimation of $\|G_{R_k}\|$ when $\|R_j\| \gg \|B\|=1$ for some $j<k$. 
Therefore, following \cite{Aihara_2019}, we focus on the error terms related to the residual norm, which can be reduced by a residual smoothing scheme, rather than on the approximation norms. 
Note that, similarly to the single right-hand side case, \eqref{gap1} is not necessarily sharp for estimating $\|G_{R_k}\|$ in the actual computation.

\Cref{Theorem1} can specifically be applied to Lanczos-type solvers, such as Gl-BiCGSTAB, Gl-CGS2, and Bl-BiCGGR. 
Note that, for example, in Gl-BiCGSTAB and Bl-BiCGGR, the approximation and residual are updated in two parts (the BiCG part and polynomial part) at each iteration by using the forms \eqref{X_R}. 
Refer to the paragraph after the proof of \Cref{Theorem2} for details; also see \cref{sec5.2}.

Moreover, we can identify the cause of a large residual gap for the recursion formulas~\eqref{B_X_R}, similarly to \eqref{gap1}, under the assumption that the columns of $Q_k \in \mathbb{R}^{n\times s}$, the range of which is equal to that of the direction matrix $P_k$, are orthonormal. 
If the number of right-hand sides $s$ is small and the iteration process is sufficiently stable, the orthonormalization of the columns of $P_k$ can be skipped. 
However, it is not clear whether large relative residual norms are the main factors causing the large residual gaps with no assumption about the orthonormality of the columns of $Q_k$, and we leave this problem open.

\begin{theorem}\label{Theorem2}
Let $X_k \in \mathbb{F}^{n\times s}$ and $R_k \in \mathbb{F}^{n\times s}$ be the $k$th approximation and residual, respectively, generated by \eqref{B_X_R} in finite precision arithmetic. 
If the columns of $Q_j\in \mathbb{R}^{n\times s}$ are (exactly) orthonormal for all $j < k$, the norm of the residual gap $G_{R_k} := (B - AX_k) - R_k$ is bounded as follows: 
\begin{align}
\|G_{R_k}\| \leq k(3+4s\sqrt{s} + 2m\sqrt{s}){\bf u}\|A\|\max_{0< j\leq k} \|X_j\| + 3(k+1){\bf u}\max_{0\leq j \leq k} \|R_j\|. \label{Bgap1} 
\end{align}
\end{theorem}
\begin{proof}
Noting that $\|Q_j\| = \sqrt{s}$ and $\|Q_j\|\|\alpha_j^\square\| = \sqrt{s}\|\alpha_j^\square\| = \sqrt{s}\|Q_j\alpha_j^\square\|$, the local errors in the updated approximation $X_k=\text{fl}(X_{k-1} + \text{fl}(Q_{k-1}\alpha_{k-1}^\square))$ and residual $R_k= \text{fl}(R_{k-1} - \text{fl}(\text{fl}(AQ_{k-1})\alpha_{k-1}^\square))$ can be evaluated as follows: 
\begin{align*}
X_k 
&= X_{k-1} + \text{fl}(Q_{k-1}\alpha_{k-1}^\square) + E_{X_k}',\quad \|E_{X_k}'\| \leq {\bf u} (2\|X_{k-1}\| + \|X_k\|) \\
&= X_{k-1} + Q_{k-1}\alpha_{k-1}^\square + E_{X_k}' + E_{X_k}^\square,\\
&\hphantom{=}\quad \|E_{X_k}^\square \| \leq s {\bf u} \|Q_{k-1}\| \|\alpha_{k-1}^\square\| = s\sqrt{s} {\bf u} \|Q_{k-1}\alpha_{k-1}^\square\| \leq s\sqrt{s} {\bf u} (\|X_{k-1}\| + \|X_k\|) \\
&= X_{k-1} + Q_{k-1}\alpha_{k-1}^\square + E_{X_k},\quad \|E_{X_k}\| \leq {\bf u} [(2+s\sqrt{s})\|X_{k-1}\| + (1+s\sqrt{s})\|X_k\|], \\
R_k 
&= R_{k-1} - \text{fl}(\text{fl}(AQ_{k-1})\alpha_{k-1}^\square) - E_{R_k}',\quad \|E_{R_k}'\| \leq {\bf u} (2\|R_{k-1}\| + \|R_k\|) \\
&= R_{k-1} - \text{fl}(AQ_{k-1})\alpha_{k-1}^\square - E_{R_k}' - E_{R_k}^\square,\\
&\hphantom{=}\quad \|E_{R_k}^\square\| \leq s {\bf u} \|AQ_{k-1}\| \|Q_{k-1}\alpha_{k-1}^\square\| \leq s\sqrt{s} {\bf u} \|A\| (\|X_{k-1}\| + \|X_k\|) \\
&= R_{k-1} - AQ_{k-1}\alpha_{k-1}^\square - E_{R_k}' - E_{R_k}^\square - E_{R_k}'''\alpha_{k-1}^\square,\\
&\hphantom{=}\quad \|E_{R_k}'''\alpha_{k-1}^\square\| \leq m{\bf u} \|A\|\|Q_{k-1}\|\|\alpha_{k-1}^\square\|\leq m \sqrt{s}{\bf u} \|A\|(\|X_{k-1}\| + \|X_k\|) \\
&= R_{k-1} - AQ_{k-1}\alpha_{k-1}^\square - E_{R_k},\\
&\hphantom{=}\quad \|E_{R_k}\| \leq {\bf u} (2\|R_{k-1}\| + \|R_k\|) + (s+m)\sqrt{s}{\bf u} \|A\|(\|X_{k-1}\| + \|X_k\|), 
\end{align*}
where $E_{X_k} := E_{X_k}'+E_{X_k}^\square$ and $E_{R_k}:=E_{R_k}' + E_{R_k}^\square + E_{R_k}'''\alpha_{k-1}^\square$. 
Thus, similarly to the evaluation of \eqref{gap_Th1}, the norm of the residual gap can be bounded as follows: 
\begin{align*}
\|(B-AX_{k}) - R_k\| \leq (3+4s\sqrt{s} + 2m\sqrt{s}){\bf u}\|A\|\sum_{j=1}^k\|X_j\| + 3{\bf u}\sum_{j=0}^k\|R_j\|. 
\end{align*}
The proof is completed by bounding the terms in the first and second sums with $k\max_j \|X_j\|$ and $(k+1)\max_j \|R_j\|$, respectively.
\end{proof}

The above discussions are useful for evaluating the residual gap in a specific method, namely Bl-BiCGSTAB, in combination with a strategy for orthonormalizing the columns of the direction matrices \cite[Algorithm~2.1]{Nakamura_2012}; this method is referred to as Bl-BiCGSTABpQ. 
The method uses the forms \eqref{B_X_R} and \eqref{X_R} in the BiCG and polynomial parts, respectively, in each iteration as follows: 
\begin{align}
\begin{split}
&\text{(BiCG part)}\quad X_k' := X_k + Q_k\alpha_k^\square,\quad R_k' := R_k - (AQ_k)\alpha_k^\square,\\
&\text{(polynomial part)}\quad X_{k+1} = X_k' + \omega_k R_k',\quad R_{k+1} = R_k' - \omega_k (AR_k'), 
\end{split}\label{G_B_X_R}
\end{align}
where $\omega_k \in \mathbb{R}$, $\alpha_k^\square \in \mathbb{R}^{s\times s}$, and $Q_k$ is a column-orthonormal matrix. 
Thus, we can evaluate the residual gap using \Cref{Theorem1,Theorem2}, as follows.

\begin{corollary}\label{Corollary1}
Let $X_k \in \mathbb{F}^{n\times s}$ and $R_k \in \mathbb{F}^{n\times s}$ be the $k$th approximation and residual, respectively, generated by Bl-BiCGSTABpQ with \eqref{G_B_X_R} in finite precision arithmetic. 
If the columns of $Q_j$ are (exactly) orthonormal for all $j < k$, then the norm of the residual gap $G_{R_k} := (B - AX_k) - R_k$ is bounded as follows: 
\begin{align}
\begin{split}
\|G_{R_k}\| 
&\leq k(4 + 2s\sqrt{s} + m + m\sqrt{s}){\bf u}\|A\| \left(\max_{0\leq j <k}\|X_j'\| + \max_{0< j \leq k}\|X_j\|\right) \\
&\hphantom{\leq}\quad + 4(k+1){\bf u}\left(\max_{0\leq j <k}\|R_j'\| +  \max_{0\leq j \leq k}\|R_j\|\right). 
\end{split}\label{Bgap2} 
\end{align}
\end{corollary}
\begin{proof}
As in the proofs of \Cref{Theorem1,Theorem2}, we can obtain the bound of $\|(B-AX_{k}) - R_k\|$ by evaluating the local errors in $X_k$, $R_k$, $X_{k-1}'$, and $R_{k-1}'$. 
\end{proof}

{\it Exact} orthonormality of the columns of $Q_j$ is required to evaluate the local errors in \Cref{Theorem2,Corollary1}. 
This assumption is not satisfied in finite precision arithmetic owing to rounding errors. 
However, the orthonormalization of the columns of $P_j$ can be performed in a backward stable manner, e.g., by using the Householder transformation \cite[Theorem~19.4]{Higham_2002} and Givens rotations \cite[Theorem~19.10]{Higham_2002}, and the evaluations \eqref{Bgap1} and \eqref{Bgap2} can almost capture the actual computations.

\section{Residual smoothing for global and block methods}\label{sec3}

In this section, we present residual smoothing for global and block Lanczos-type solvers. 
All of the discussions in \cref{sect_standard,sect_interactive} assume exact arithmetic.

\subsection{Simple smoothing scheme}\label{sect_standard}

We first present a simple residual smoothing scheme, which is a naive extension of the classical smoothing technique in \cite{Schonauer_1987,Weiss_1996} to the case of multiple right-hand sides (cf.~\cite{Zhang_2008}).

Let $\{X_k\}$ and $\{R_k\}$ be the primary sequences of the approximations and residuals, respectively, obtained by a global or block method. 
Then, new sequences of approximations $Y_k$ and the corresponding smoothed residuals $S_k := B - AY_k$ are generated by 
\begin{align}
Y_k = (1-\eta_k) Y_{k-1} + \eta_kX_k,\quad S_k = (1-\eta_k)S_{k-1} + \eta_kR_k, \label{S-W}
\end{align}
where $Y_0 := X_0$ is the initial guess, $S_0 := R_0$ is the initial residual, and $\eta_k \in \mathbb{R}$ is a smoothing parameter. 
Based on the typical strategy, we select the parameter $\eta_k$ such that the updated smoothed residual norm $\|S_k\|$ is locally minimized; that is, 
\begin{align}
\eta_k := - \frac{\langle S_{k-1}, R_k - S_{k-1}\rangle_F}{\| R_k - S_{k-1} \|^2}. \label{smo_eta}
\end{align}
Thus, we obtain $\|S_k\| \leq \|R_k\|$ and $\|S_k\| \leq \|S_{k-1}\|$. 
The former inequality implies that $\|S_k\|$ decreases no slower than $\|R_k\|$; however, as is well known in the classical residual smoothing scheme, the smoothed residual will not converge much faster than the primary residual in generic cases; for example, see \cite{Gutknecht_2001_2,Walker_1995}.
The latter inequality implies that $\|S_k\|$ decreases monotonically and plays an important role in reducing the residual gap in CIRS.

We note that the smoothing parameter $\eta_k$ can be selected using a different technique, such as quasi-minimal residual smoothing \cite{Zhou_1994}. 
With this scheme, if $\|S_k\|$ decreases smoothly, we can obtain numerical results that are similar to the case of selecting \eqref{smo_eta}. 
As the effects of CIRS are not significantly dependent on the parameter itself, we do not elaborate on its selection in this study.

\subsection{Cross-interactive smoothing scheme}\label{sect_interactive}

We present a cross-interactive residual smoothing scheme. 
The original CIRS that was proposed in \cite{Aihara_2019} is an improvement of the alternative smoothing scheme that was introduced in \cite{Zhou_1994}, whereby the primary and smoothed sequences influence one another in the iteration process. 
We apply this concept to the case of solving systems with multiple right-hand sides \eqref{AX=B}.

In the following, the variables used in CIRS are displayed with a hat symbol `$\ \hat{}\ $'. 
The smoothing process begins with the computation of the direction matrix $\hat P_k$ that corresponds to $X_{k+1}-X_k$ using the primary method. 
Then, we compute auxiliary matrices $\hat V_{k+1}$ and $\hat U_{k+1}$ as follows: 
\begin{align}
\hat V_{k+1} = (1-\hat \eta_k) \hat V_k + \hat P_k,\quad \hat U_{k+1} = A\hat V_{k+1}, \label{CIRS_V_U}
\end{align}
where $\hat \eta_k \in \mathbb{R}$ is a smoothing parameter, $\hat \eta_0 := 0$, $\hat V_0 := O$, and $A\hat V_{k+1}$ is obtained by explicitly multiplying $\hat V_{k+1}$ by $A$.
We next compute new approximations and corresponding smoothed residuals recursively as follows: 
\begin{align}
\hat Y_{k+1} = \hat Y_k + \hat \eta_{k+1}\hat V_{k+1},\quad \hat S_{k+1} = \hat S_k - \hat \eta_{k+1}\hat U_{k+1}, \label{CIRS_Y_S}
\end{align}
where $\hat Y_0 = X_0$, $\hat S_0 = R_0$, and the smoothing parameter is selected as $\hat \eta_{k+1} = \break \langle \hat S_k, \hat U_{k+1}\rangle_F / \langle \hat U_{k+1}, \hat U_{k+1} \rangle_F$ like \eqref{smo_eta}. 
Finally, we generate the following: 
\begin{align}
X_{k+1} = \hat Y_{k+1} + (1-\hat \eta_{k+1})\hat V_{k+1},\quad R_{k+1} = \hat S_{k+1} - (1-\hat \eta_{k+1})\hat U_{k+1}. \label{CIRS_X_R}
\end{align}
Here, we note that $X_{k+1}$ and $R_{k+1}$ given by \eqref{CIRS_X_R} coincide with  the primary approximation and residual, respectively, as shown in \Cref{Prop1} below. 
This follows naturally from \cite[Lemma~2.1]{Aihara_2019}.

\begin{proposition}\label{Prop1}
Let $\{X_k\}$ and $\{R_k\}$ be the primary sequences of the approximations and residuals, respectively, which are generated by \eqref{X_R} or \eqref{B_X_R}, and let $\hat P_k := X_{k+1}-X_k$ be the direction matrix. 
Then, for the iteration matrices $\hat S_k$, $\hat U_k$ ($=A\hat V_k$), $\hat V_k$, and $\hat Y_k$ that are generated by \eqref{CIRS_V_U} and \eqref{CIRS_Y_S}, the identities 
\begin{align}
\hat Y_k + (1-\hat \eta_k)\hat V_k = X_k,\quad \hat S_k - (1-\hat \eta_k)\hat U_k = R_k \label{SI_new}
\end{align}
hold for $k=0,1,2,\dots$, where $\hat \eta_k \in \mathbb{R}$ is a smoothing parameter. 
\end{proposition}

The values $X_{k+1}$ and $R_{k+1}$ in \eqref{CIRS_X_R} are generated following the updates of the new approximation $\hat Y_{k+1}$ and corresponding smoothed residual $\hat S_{k+1}$ using the smoothed sequences, and are returned to the primary method. 
Thus, the primary and smoothed sequences influence one another, in contrast to the sequences generated by \eqref{S-W}.

Nevertheless, the following proposition shows the equivalence between CIRS and \eqref{S-W} with \eqref{smo_eta}, and can be derived from the statements in \cite{Aihara_2019,Gutknecht_2001,Zhou_1994}. 
The proof is not included in this paper, but it can easily be shown by induction.

\begin{proposition}\label{Prop2}
For given primary sequences $\{X_k\}$ and $\{R_k\}$, let $\hat Y_k$, $\hat S_k$, and $\hat \eta_k$ be the $k$th approximation, smoothed residual, and smoothing parameter, respectively, generated by \eqref{CIRS_V_U} and \eqref{CIRS_Y_S}, and let $Y_k$, $S_k$, and $\eta_k$ be those generated by \eqref{S-W} with \eqref{smo_eta}. 
Then, the identities $\hat Y_k = Y_k$ and $\hat S_k = S_k$ hold for $k=0,1,2,\dots$ and $\hat \eta_k = \eta_k$ for $k=1,2,\dots$. 
\end{proposition}

As a result, we can observe that the smoothed residual does not converge slower than the primary one $\|\hat S_k\| \leq \|R_k\|$ and that the smoothed residual norm decreases monotonically $\|\hat S_k\| \leq \|\hat S_{k-1}\|$.

\Cref{alg1} is an extension of CIRS that can be applied to global and block methods. 
If the $k$th primary approximation is updated in the form of $X_{k+1} = X_k + \alpha_k P_k$ with $\alpha_k \in \mathbb{R}$ or $X_{k+1} = X_k + P_k\alpha_k^\square$ with $\alpha_k^\square \in \mathbb{R}^{s\times s}$, then $\hat P_k$ is defined as $\alpha_kP_k$ or $P_k\alpha_k^\square$, respectively. 
In line~5, although a matrix multiplication by $A$ is required, the total number of multiplications by $A$ per iteration in the primary method and in the corresponding smoothed method can be made the same through a sophisticated formulation; for details, see \cref{sec5}.

\begin{algorithm}[t]
\caption{CIRS for global and block methods}\label{alg1}
\begin{algorithmic}[1]
\REQUIRE An initial guess $X_0$ and the initial residual $R_0:=B-AX_0$. 
\STATE Set $\hat Y_0 := X_0$, $\hat S_0 := R_0$, $\hat V_0 := O$, and $\hat \eta_0 := 0$.
\FOR{$k=0,1,2,\dots$ until convergence}
	\STATE Compute $\hat P_k$ (corresponding to $X_{k+1}-X_k$) using the primary method.
	\STATE $\hat V_{k+1} = (1-\hat \eta_k) \hat V_k + \hat P_k$
	\STATE Compute $\hat U_{k+1} = A\hat V_{k+1}$ with an explicit multiplication by $A$.
	\STATE $\hat \eta_{k+1} = \langle \hat S_k, \hat U_{k+1}\rangle_F / \langle \hat U_{k+1}, \hat U_{k+1} \rangle_F$
	\STATE $\hat Y_{k+1} = \hat Y_k + \hat \eta_{k+1}\hat V_{k+1},\quad \hat S_{k+1} = \hat S_k - \hat \eta_{k+1}\hat U_{k+1}$
	\RETURN $X_{k+1} = \hat Y_{k+1} + (1-\hat \eta_{k+1})\hat V_{k+1},\quad R_{k+1} = \hat S_{k+1} - (1-\hat \eta_{k+1})\hat U_{k+1}$
\ENDFOR
\end{algorithmic}
\end{algorithm}

\section{Advantage of CIRS}\label{sec4}

We discuss the residual gap for the presented residual smoothing schemes in finite precision arithmetic. 
We refer to \cref{sec2} for matrix operations, taking into account the rounding errors.

We first consider the simple smoothing scheme outlined in \cref{sect_standard}. 
Although smooth convergence behavior can be achieved by \eqref{S-W}, the residual gap cannot be reduced because of the following property. 
\Cref{Theorem3} below is a naive extension of several results for a single right-hand side in \cite[sections~4.1 and 5.1]{Gutknecht_2001} to the case of multiple right-hand sides. 
The proof is not included in this paper, but it follows naturally from \cite{Gutknecht_2001}.

\begin{theorem}[cf.~\cite{Gutknecht_2001}]\label{Theorem3}
Let $\{X_k\}$ and $\{R_k\}$ be the primary sequences of the approximations and residuals, respectively, which are generated by \eqref{X_R} or \eqref{B_X_R}, and let $\{Y_k\}$ and $\{S_k\}$ be the corresponding smoothed sequences that are generated by \eqref{S-W} in finite precision arithmetic. 
Then, the norm of the residual gap $G_{S_k} :=  (B - AY_k) - S_k$ is bounded as follows: 
\begin{align}
&\|G_{S_k}\| \leq |1-\eta_k|\|G_{S_{k-1}}\|+ |\eta_k| \|G_{R_k}\| +\|A\|\|E_{Y_k}\| + \|E_{S_k}\|, \label{gap2}
\end{align}
where $E_{Y_k}$ and $E_{S_k}$ are the local errors in the updated approximation $Y_k$ and residual $S_k$, respectively, and these errors satisfy 
\begin{align}
\begin{split}
\|E_{Y_k}\| \leq 3{\bf u}|1-\eta_k| \|Y_{k-1}\| + 2{\bf u}|\eta_k|\|X_k\|, \\
\|E_{S_k}\| \leq 3{\bf u}|1-\eta_k| \|S_{k-1}\| + 2{\bf u}|\eta_k|\|R_k\|. 
\end{split}\label{local1}
\end{align}
\end{theorem}

In actual computations, the residual norm $\|R_k\|$ of a Lanczos-type solver often increases significantly in the early iterations. 
According to the property of residual smoothing, $|\eta_k|$ is selected to be small when $\|R_k\|$ becomes large; that is, it can be expected from \eqref{gap2} and \eqref{local1} that the residual gap and local errors in the smoothed sequences are relatively small, although $\|R_k\| \gg \|B\| = 1$ holds. 
However, after $\|R_k\|$ increases significantly, the subsequent iterations move into a stage where $\|R_k\| \approx \|B\|$ and $|\eta_k| \approx 1$. 
Furthermore, $|\eta_k|\|G_{R_k}\|$ is the dominant part of the right-hand side in \eqref{gap2}, and $\|G_{S_k}\|$ increases drastically to the same order of magnitude as $\|G_{R_k}\|$. 
We present such a phenomenon in the numerical experiments in \cref{sec7.1}.

Next, we consider the new CIRS displayed in \Cref{alg1}. 
\Cref{Theorem4,Theorem5} are extensions of the main results for the original CIRS; that is, \cite[Eqs.~(3.12) and (3.18)]{Aihara_2019}, for the case of multiple right-hand sides.

As the recursion formulas for updating $\hat Y_k$ and $\hat S_k$ in line~7 of \Cref{alg1} have the same forms as \eqref{X_R}, we obtain the following result for the residual gap of the smoothed sequences, as in \Cref{Theorem1} for that of the primary sequences.

\begin{theorem}\label{Theorem4}
Let $\hat Y_k \in \mathbb{F}^{n\times s}$ and $\hat S_k \in \mathbb{F}^{n\times s}$ be the $k$th approximation and smoothed residual, respectively, which are generated by \Cref{alg1} in finite precision arithmetic. 
Then, the norm of the residual gap $G_{\hat S_k} := (B - A\hat Y_k) - \hat S_k$ is bounded as follows: 
\begin{align}
\|G_{\hat S_k}\| \leq k(5+2m){\bf u} \|A\| \max_{0<j\leq k} \|\hat Y_j\| + 5(k+1){\bf u} \max_{0\leq j \leq k} \|\hat S_j\|. \label{gap3} 
\end{align}
\end{theorem}

As the smoothed residual norm $\|\hat S_k\|$ decreases monotonically, the upper bound in \eqref{gap3} is relatively small compared to that in \eqref{gap1}, and the residual gap is expected to be reduced. 
Moreover, as the primary and smoothed sequences influence one another in \Cref{alg1}, the evaluation of the residual gap for the primary sequences is also improved over \eqref{gap1} as follows:

\begin{theorem}\label{Theorem5}
Let $X_k \in \mathbb{F}^{n\times s}$ and $R_k \in \mathbb{F}^{n\times s}$ be the $k$th approximation and residual, respectively, which are generated by \Cref{alg1} in finite precision arithmetic. 
Then, the norm of the residual gap $G_{R_k} := (B - AX_k) - R_k$ is bounded as follows: 
\begin{align}
\begin{split}
&\|G_{R_k}\| \leq k(5+2m){\bf u} \|A\| \max_{0<j\leq k} \|\hat Y_j\| + 5(k+1) {\bf u} \max_{0\leq j \leq k} \|\hat S_j\|\\
&\hphantom{\|G_{R_k}\|\leq}\quad + (2+m){\bf u} \|A\| \|X_k\| + 2{\bf u} \|R_k\|.
\end{split} \label{gap4}
\end{align}
\end{theorem}
\begin{proof}
Similarly to the proof of \Cref{Theorem1}, the local errors in $\hat Y_k$, $\hat S_k$, $X_k$, and $R_k$ can be evaluated as follows: 
\begin{align*}
&\hat Y_k = \hat Y_{k-1} + \hat \eta_k \hat V_k + E_{\hat Y_k},\quad \|E_{\hat Y_k}\| \leq {\bf u} (3\|\hat Y_{k-1}\| + 2\|\hat Y_k\|), \\
&\hat S_k = \hat S_{k-1} - \hat \eta_k A \hat V_k - E_{\hat S_k},\quad \|E_{\hat S_k}\| \leq {\bf u} (3\|\hat S_{k-1}\| + 2\|\hat S_k\|) + m{\bf u} \|A\|(\|\hat Y_{k-1}\| + \|\hat Y_k\|), \\
&X_k = \hat Y_k + \hat \zeta_k \hat V_k + E_{X_k},\quad \|E_{X_k}\| \leq {\bf u} (3\|\hat Y_k\| + 2\|X_k\|), \\
&R_k = \hat S_k - \hat \zeta_k A \hat V_k - E_{R_k},\quad \|E_{R_k}\| \leq {\bf u} (3\|\hat S_k\| + 2\|R_k\|) + m{\bf u} \|A\|(\|\hat Y_k\| + \|X_k\|), 
\end{align*}
where $\hat \zeta_k := \text{fl}(1-\hat \eta_k)$. 
Subsequently, the norm of the residual gap can be bounded as follows: 
\begin{align*}
&\|(B-AX_{k}) - R_k\|\\
&\quad \leq \|(B-A\hat Y_k) - \hat S_k\| +  \|A\|\|E_{X_k}\| + \|E_{R_k}\|\\
&\quad \leq \|(B-A\hat Y_{k-1}) - \hat S_{k-1}\| + \|A\|\|E_{\hat Y_k}\| + \|E_{\hat S_k}\| +  \|A\|\|E_{X_k}\| + \|E_{R_k}\|\\
&\quad\leq \|A\|\sum_{j=1}^k\|E_{\hat Y_j}\| + \sum_{j=1}^k\|E_{\hat S_j}\|  + \|A\|\|E_{X_k} \|+ \|E_{R_k}\|\\
&\quad\leq (5+2m){\bf u}\|A\|\sum_{j=1}^k\|\hat Y_j\| + 5{\bf u}\sum_{j=0}^k\|\hat S_j\|+ (2+m){\bf u} \|A\| \|X_k\| + 2{\bf u} \|R_k\|.
\end{align*}
The proof is completed by bounding the terms in the first and second sums with $k\max_j \|\hat Y_j\|$ and $(k+1)\max_j \|\hat S_j\|$, respectively.
\end{proof}

The upper bound in \eqref{gap4} is dependent on $\max_j \|\hat Y_j\|$ and $\max_j\|\hat S_j\|$ for $j \leq k$, and only the $k$th norms $\|X_k\|$ and $\|R_k\|$. 
Therefore, similarly to the case of the single right-hand side \cite{Aihara_2019}, $\|G_{R_k}\|$ can increase when $\|R_k\|$ increases, but it can be reduced as $\|R_j\|$ decreases for $j > k$. 
The upper bounds in \eqref{gap3} and \eqref{gap4} of the final residual gap are of the same order of magnitude when the primary and smoothed sequences converge, and $\hat Y_k$ and $X_k$ in \Cref{alg1} are expected to attain the same level of accuracy.

\section{Specific smoothed algorithms}\label{sec5}

In this section, we apply CIRS to Gl-CGS2 \cite{Zhang_2010}, Gl-BiCGSTAB \cite{Jbilou_2005}, Bl-BiCGSTAB \cite{ElGuennouni_2003}, and Bl-BiCGGR \cite{Tadano_2009}. 
For the block methods, we consider their stabilized variants that are used with strategies that orthonormalize the columns of the iteration matrices, namely Bl-BiCGSTABpQ \cite[Algorithm~2.1]{Nakamura_2012} and Bl-BiCGGRrQ \cite[Fig.~2]{Tadano_2017}.

\subsection{Smoothed variant of Gl-CGS2}

As Gl-CGS2 uses the forms \eqref{X_R}, we can apply CIRS directly by setting $\hat P_k := \alpha_k U_k + \tilde \alpha_k S_k$ in \Cref{alg1}, where the matrix $\alpha_k U_k + \tilde \alpha_k S_k$ is provided in \cite[line~13 of Algorithm~5]{Zhang_2010}.

\begin{algorithm}[t]
\caption{Smoothed global CGS2 (S-Gl-CGS2)}\label{alg2}
\begin{algorithmic}[1]
\STATE Select an initial guess $X$ and compute $R = B-AX$.
\STATE Set $\hat Y := X$, $\hat S := R$, $\hat V := O$, and $\hat \zeta := 0$.
\STATE Select $R_0^{\bullet}$ and $R_0^{\circ}$, and compute $Z_0^{\bullet} = A^\top R_0^{\bullet}$ and $Z_0^{\circ} = A^\top R_0^{\circ}$.
\STATE Set $P := R$, $U := R$, and $T := R$.
\WHILE{$\|\hat S\| > tol$}
	\STATE $V = AP,\quad \sigma^{\bullet} = \langle R_0^{\bullet}, V\rangle_F,\quad \sigma^{\circ} = \langle R_0^{\circ}, V\rangle_F$
	\STATE $\alpha^{\bullet} = \langle R_0^{\bullet}, R\rangle_F / \sigma^{\bullet},\quad \alpha^{\circ} = \langle R_0^{\circ}, R\rangle_F / \sigma^{\circ}$
	\STATE $W = T - \alpha^{\bullet} V,\quad Q = U - \alpha^{\circ} V$
	\STATE $\hat P = \alpha^{\bullet}U + \alpha^{\circ} W,\quad \hat V = \hat \zeta \hat V + \hat P,\quad \hat U = A \hat V$
	\STATE $\hat \eta = \langle \hat S, \hat U\rangle_F / \langle \hat U, \hat U \rangle_F,\quad \hat \zeta = 1-\hat \eta$
	\STATE $\hat Y = \hat Y + \hat \eta \hat V,\quad \hat S = \hat S - \hat \eta \hat U$
	\STATE $X = \hat Y + \hat \zeta \hat V,\quad R = \hat S - \hat \zeta \hat U$
	\STATE $\beta^{\bullet} = \langle Z_0^{\bullet}, W\rangle_F / \sigma^{\bullet},\quad \beta^{\circ} = \langle Z_0^{\circ}, Q\rangle_F / \sigma^{\circ}$
	\STATE $U = R - \beta^{\bullet} Q,\quad T = R - \beta^{\circ} W,\quad P = T - \beta^{\bullet} (Q - \beta^{\circ} P)$
\ENDWHILE
\end{algorithmic}
\end{algorithm}

\Cref{alg2} displays the resulting smoothed Gl-CGS2 method (S-Gl-CGS2), where $\alpha_k U_k + \tilde \alpha_k S_k$ is renamed as $\alpha^{\bullet} U + \alpha^{\circ} W$. 
Lines~9--12 of \Cref{alg2} correspond to CIRS and the other lines follow from Gl-CGS2. 
In line~9, we use an explicit multiplication by $A$, but we do not need to compute $A\hat P_k$, which is required in Gl-CGS2. 
Thus, the smoothed variant can be implemented without additional multiplications by $A$. 
Note that Gl-CGS2 \cite[Algorithm~5]{Zhang_2010} uses four multiplications by $A$ per iteration, but two of these can be reduced by computing $A^\top R_0^{\bullet}$ and $A^\top R_0^{\circ}$ in advance and storing them if an operator $A^\top$ is available, and this efficient approach is also incorporated into \Cref{alg2}. Here, $R_0^{\bullet}$ and $R_0^{\circ}$ are the initial shadow residuals used in Gl-CGS2. 
The initial shadow residual $R_0^{\bullet}$ is a starting matrix to construct $\mathcal K_k^G(A^\top, R_0^{\bullet})$, which is used in a bi-orthogonality condition for the underlying Gl-BiCG residuals. 
Many global Lanczos-type solvers, such as Gl-CGS and Gl-BiCGSTAB, require one initial shadow residual, but Gl-CGS2 uses another shadow residual $R_0^{\circ}$ to determine the so-called stabilizing polynomials. 
We refer the reader to \cite{Fokkema_1996,Jbilou_2005,Zhang_2010} for details.

Note that the derivation of Gl-CGS2 in \cite{Zhang_2010} is different from that of the original CGS2 for a single linear system in \cite{Fokkema_1996}. 
Another Gl-CGS2 algorithm and its smoothed variant can be derived based on \cite{Fokkema_1996}, and this approach also requires no additional multiplications by $A$. 
However, as our main purpose is to illustrate the effectiveness of CIRS, we do not discuss alternative algorithms of the primary method further.

\begin{algorithm}[t]
\caption{Smoothed global BiCGSTAB (S-Gl-BiCGSTAB)}\label{alg3}
\begin{algorithmic}[1]
\STATE Select an initial guess $X$ and compute $R = B-AX$.
\STATE Set $\hat Y := X$, $\hat S := R$, $\hat V := O$, and $\hat \zeta := 0$.
\STATE Select $R_0^{\bullet}$ and compute $Z_0^{\bullet} = A^\top R_0^{\bullet}$.
\STATE Set $P := R$, $R' := O$, and $\omega := 0$.
\WHILE{$\|\hat S\| > tol$}
	\STATE $\sigma = \langle Z_0^{\bullet}, P\rangle_F,\quad \alpha = \langle R_0^{\bullet}, R\rangle_F / \sigma$
	\STATE $\hat P = \omega R' + \alpha P,\quad \hat V = \hat \zeta \hat V + \hat P,\quad \hat U = A \hat V$
	\STATE $\hat \eta = \langle \hat S, \hat U\rangle_F / \langle \hat U, \hat U \rangle_F,\quad \hat \zeta = 1-\hat \eta$
	\STATE $\hat Y = \hat Y + \hat \eta \hat V,\quad \hat S = \hat S - \hat \eta \hat U$
	\STATE $X' = \hat Y + \hat \zeta \hat V,\quad R' = \hat S - \hat \zeta \hat U$
	\STATE $V = (R - R')/\alpha,\quad T = A R',\quad \omega = \langle R', T\rangle_F / \langle T, T \rangle_F$
	\STATE $X = X' + \omega R',\quad R = R' - \omega T$
	\STATE $\beta = \langle R_0^{\bullet}, T\rangle_F / \sigma,\quad P = R - \beta (P - \omega V)$
\ENDWHILE
\end{algorithmic}
\end{algorithm}

\subsection{Smoothed variant of Gl-BiCGSTAB}\label{sec5.2}

Gl-BiCGSTAB updates two residuals in the forms \eqref{X_R}: $R_k' := R_k - \alpha_k A P_k$ in the BiCG part and $R_{k+1} = R_k' - \omega_k AR_k'$ in the polynomial part, where $\alpha_k = \langle R_0^{\bullet}, R_k\rangle_F / \langle R_0^{\bullet}, AP_k\rangle_F$ with an iteration matrix $P_k$ and $\omega_k = \langle R_k', AR_k'\rangle_F / \langle AR_k', AR_k'\rangle_F$.
Therefore, we can apply CIRS to both the BiCG and polynomial parts. 
However, this requires two additional multiplications by $A$ per iteration. 
To circumvent this issue, based on \cite[section~4]{Aihara_2019}, we reformulate the updating process so that no additional multiplications by $A$ are required; see also \cite[section~4.6]{Sleijpen_1996}. 
We consider the recursion formulas of the approximation and residual in the BiCG part, 
\begin{align*}
X_k' = X_{k-1}' + (\omega_{k-1} R_{k-1}' + \alpha_k P_k), \quad R_k' = R_{k-1}' - A(\omega_{k-1} R_{k-1}' + \alpha_k P_k), 
\end{align*}
and perform CIRS by setting $\hat P_k: = \omega_{k-1} R_{k-1}' + \alpha_k P_k$. 
Thereafter, we compute $\alpha_k =\break \langle R_0^{\bullet}, R_k\rangle_F/\langle Z_0^{\bullet}, P_k\rangle_F$, where $Z_0^{\bullet} := A^\top R_0^{\bullet}$ is computed and stored in advance. 
To update $P_k$ to $P_{k+1}$ via $AP_k$, we use the backward formulation $AP_k := (R_k - R_k')/\alpha_k$ following the computation of $R_k'$. 
The resulting smoothed Gl-BiCGSTAB method (S-Gl-BiCGSTAB), which requires no additional multiplications by $A$, is presented in \Cref{alg3}.

\subsection{Smoothed variant of Bl-BiCGSTABpQ}

We apply CIRS to Bl-BiCGSTABpQ using the same approach as in \cref{sec5.2}. 
We rewrite \eqref{G_B_X_R} as 
\begin{align*}
X_k' = X_{k-1}' + (\omega_{k-1} R_{k-1}' + Q_k\alpha_k^\square), \quad R_k' = R_{k-1}' - A(\omega_{k-1} R_{k-1}' + Q_k\alpha_k^\square), 
\end{align*}
and set $\hat P_k: = \omega_{k-1} R_{k-1}' + Q_k\alpha_k^\square$ in \Cref{alg1}. 
The matrix $\alpha_k^\square \in \mathbb{R}^{s\times s}$ is obtained by solving an $s$-dimensional linear system $({Z_0^{\bullet}}^\top Q_k) \alpha_k^\square ={R_0^{\bullet}}^\top R_k$, where $Z_0^{\bullet} := A^\top R_0^{\bullet}$ is computed and stored in advance. 
Following the computation of $R_k'$, the matrix $AQ_k$ is provided as a solution of the system $(AQ_k) \alpha_k^\square = R_k - R_k'$. 
The resulting smoothed Bl-BiCGSTABpQ method (S-Bl-BiCGSTABpQ), which also requires no additional multiplications by $A$, is displayed in \Cref{alg4}. 
Note that $\textbf{qf}(\cdot)$ in line~6 denotes the Q-factor of the QR factorization of a matrix.

\subsection{Smoothed variant of Bl-BiCGGRrQ}

Finally, we apply CIRS to Bl-BiCGGRrQ, which is a stabilized variant of Bl-BiCGGR.

The original Bl-BiCGGR aims to reduce the residual gap of Bl-BiCGSTAB. 
The basic concept is to reformulate the recursion formulas as follows: 
\begin{align*} 
X_{k+1} 
&= X_k' + \omega_k R_k' = X_k + P_k\alpha_k^\square + \omega_kR_k - \omega_k AP_k\alpha_k^\square\\
&= (X_k + \omega_k R_k) + U_k,\quad U_k:=(P_k -\omega_k AP_k)\alpha_k^\square,\\
R_{k+1} &= [R_k - \omega_k(AR_k)] - AU_k. 
\end{align*}
Thus, Bl-BiCGGR can use the forms \eqref{X_R}. 
However, as discussed in \cref{sec2}, the residual gap may become large when the maximum of the residual norms is relatively large. 
Note that $\omega_k \in \mathbb{R}$ is determined by minimizing $\|R_k - \omega_k(AR_k)\|$ instead of $\|R_k' - \omega_k(AR_k')\|$.

\begin{algorithm}[t]
\caption{Smoothed block BiCGSTABpQ (S-Bl-BiCGSTABpQ)}\label{alg4}
\begin{algorithmic}[1]
\STATE Select an initial guess $X$ and compute $R = B-AX$.
\STATE Set $\hat Y := X$, $\hat S := R$, $\hat V := O$, and $\hat \zeta := 0$.
\STATE Select $R_0^{\bullet}$ and compute $Z_0^{\bullet} = A^\top R_0^{\bullet}$.
\STATE Set $P := R$, $R' := O$, and $\omega := 0$.
\WHILE{$\|\hat S\| > tol$}
	\STATE $Q = \textbf{qf}(P),\quad \sigma = {Z_0^{\bullet}}^\top Q$
	\STATE Solve $\sigma \alpha = {R_0^{\bullet}}^\top R$ for $\alpha$.
	\STATE $\hat P = \omega R' + Q\alpha,\quad \hat V = \hat \zeta \hat V + \hat P,\quad \hat U = A \hat V$
	\STATE $\hat \eta = \langle \hat S, \hat U\rangle_F / \langle \hat U, \hat U \rangle_F,\quad \hat \zeta = 1-\hat \eta$
	\STATE $\hat Y = \hat Y + \hat \eta \hat V,\quad \hat S = \hat S - \hat \eta \hat U$
	\STATE $X' = \hat Y + \hat \zeta \hat V,\quad R' = \hat S - \hat \zeta \hat U$
	\STATE Solve $V \alpha = R - R'$ for $V$.
	\STATE $T = A R',\quad \omega = \langle R', T\rangle_F / \langle T, T \rangle_F$
	\STATE $X = X' + \omega R',\quad R = R' - \omega T$
	\STATE Solve $\sigma \beta = {R_0^{\bullet}}^\top T$ for $\beta$.
	\STATE $P = R - (Q - \omega V)\beta$
\ENDWHILE
\end{algorithmic}
\end{algorithm}

\begin{algorithm}[t]
\caption{Smoothed block BiCGGRrQ (S-Bl-BiCGGRrQ)}\label{alg5}
\begin{algorithmic}[1]
\STATE Select an initial guess $X$ and compute $R = B-AX$.
\STATE Set $\hat Y := X$, $\hat S := R$, $\hat V := O$, and $\hat \zeta := 0$.
\STATE Select $R_0^{\bullet}$ and compute $[Q, \xi] = \textbf{qr}(R)$, $W = AQ$, and $\rho_- = {R_0^{\bullet}}^\top Q$.
\STATE Set $P := Q$ and $V := W$.
\WHILE{$\|\hat S\| > tol$}
	\STATE Solve $({R_0^{\bullet}}^\top V) \alpha = \rho_-$ for $\alpha$.
	\STATE $\omega = \langle W, Q\rangle_F / \langle W, W \rangle_F, \quad U = (P - \omega V)\alpha$	
	\STATE $\hat P = (\omega Q + U)\xi,\quad \hat V = \hat \zeta \hat V + \hat P,\quad \hat U = A\hat V$
	\STATE $\hat \eta = \langle \hat S, \hat U\rangle_F / \langle \hat U, \hat U \rangle_F,\quad \hat \zeta = 1-\hat \eta$
	\STATE $\hat Y = \hat Y + \hat \eta \hat V,\quad \hat S = \hat S - \hat \eta \hat U$
	\STATE $X = \hat Y + \hat \zeta \hat V,\quad R = \hat S - \hat \zeta \hat U$
	\STATE Solve $T \xi = R$ for $T$.
	\STATE $Z = Q - \omega W - T$
	\STATE $[Q, \xi] = \textbf{qr}(R),\quad W = AQ,\quad \rho_+ = {R_0^{\bullet}}^\top Q$
	\STATE Solve $\rho_- \gamma = \rho_+ / \omega$ for $\gamma$.
	\STATE $\rho_- = \rho_+,\quad P = Q + U\gamma,\quad V = W + Z\gamma$
\ENDWHILE
\end{algorithmic}
\end{algorithm}

However, as Bl-BiCGGR exhibits numerical instabilities for a large $s$, a stabilized variant Bl-BiCGGRrQ, in which the columns of $R_k$ are orthonormalized, has also been developed. 
This method updates the approximation using the form 
\begin{align*} 
X_{k+1} &= X_k + (\omega_k Q_k + U_k) \xi_k, 
\end{align*}
where $Q_k \xi_k$ corresponds to the QR factorization of the residual $R_k$. 
In the actual computation, the residual is not computed explicitly; instead, the column-orthonormal matrix $Q_k$ and upper triangular matrix $\xi_k$ are computed by 
\begin{align*} 
[Q_{k+1}, \tau_{k+1}] = \textbf{qr}(Q_k - \omega_k AQ_k -AU_k),\quad \xi_{k+1} = \tau_{k+1} \xi_k, 
\end{align*}
where \textbf{qr}($\cdot$) denotes the QR factorization of a matrix, and the first and second values of \textbf{qr}($\cdot$) are the Q- and R-factors, respectively. 
For details, we refer the reader to \cite{Tadano_2017}.

Subsequently, we apply CIRS to Bl-BiCGGRrQ by setting $\hat P_k: = (\omega_k Q_k + U_k) \xi_k$. 
The resulting smoothed Bl-BiCGGRrQ method (S-Bl-BiCGGRrQ), which also requires no additional multiplications by $A$, is presented in \Cref{alg5}. 
 
As $R_{k+1}$ is explicitly obtained instead of the matrix $Q_k - \omega_k AQ_k -AU_k$ in CIRS, we solve the system $T_{k+1} \xi_k = R_{k+1}$ to obtain the matrix $T_{k+1} := Q_{k+1} \tau_{k+1}$ in line~12, and compute $Z_k := AU_k = Q_k - \omega_k W_k - T_{k+1}$ in line~13, where $W_k := AQ_k$.

\section{Application of global methods to Sylvester equation}\label{sec6}

Global methods can easily be applied to general linear matrix equations \cite{Beik_2011,Jbilou_1999,Zhang_2010}. 
We describe how to apply the presented global algorithms to the Sylvester equation 
\begin{align}
AX - XC = B, \label{AX-XC=B}
\end{align}
where $A \in \mathbb{R}^{n\times n}$, $C \in \mathbb{R}^{s\times s}$, $B := [\bm{b}_1, \bm{b}_2, \dots, \bm{b}_s] \in \mathbb{R}^{n\times s}$, and $X := [\bm{x}_1, \bm{x}_2, \dots, \bm{x}_s] \in \mathbb{R}^{n\times s}$ with $s \ll n$. Moreover, we discuss the residual gap when solving \eqref{AX-XC=B}.

The Sylvester equation \eqref{AX-XC=B} can be represented by a standard linear system $\tilde A \bm{x} = \bm{b}$, where $\tilde A := I_s \otimes A - C^\top \otimes I_n \in \mathbb{R}^{ns\times ns}$ and $\bm{b} := [\bm{b}_1^\top,\bm{b}_2^\top,\dots,\bm{b}_s^\top]^\top$. 
Subsequently, the linear transformation by $\tilde A$ for an arbitrary $\bm{v}:= [\bm{v}_1^\top,\bm{v}_2^\top,\dots,\bm{v}_s^\top]^\top \in \mathbb{R}^{ns}$, $\bm{v}_i\in \mathbb{R}^n$ can be expressed by a linear operator $\mathcal A$ that is defined as $\mathcal A(V) := AV - VC$, where $V := [\bm{v}_1,\bm{v}_2,\dots,\bm{v}_s]$. 
Therefore, the application of the standard Krylov subspace methods to $\tilde A \bm{x} = \bm{b}$ corresponds to the application of their global counterparts to $\mathcal A(X) = B$. 
Such global methods can be implemented by replacing the multiplications with $A$ in \eqref{AX=B} by the transformations with $\mathcal A$; for example, the initial residual is defined as $R_0:=B-\mathcal A(X_0) = B-(AX_0 - X_0C)$. 
Moreover, as $\tilde A^\top = I_s \otimes A^\top - C \otimes I_n$, the adjoint of $\mathcal A$ with respect to the inner product $\langle \cdot, \cdot \rangle_F$ can be defined as $\mathcal A^\top(V) := A^\top V - V C^\top$.

At this point, we reconsider the recursion formulas \eqref{X_R} used in the global methods. 
When solving \eqref{AX-XC=B}, the recursions \eqref{X_R} are replaced with 
\begin{align}
X_{k+1} = X_k + \alpha_k P_k,\quad R_{k+1} = R_k - \alpha_k \mathcal A(P_k) = R_k - \alpha_k (AP_k - P_kC), \label{X_R_Sylv}
\end{align}
where $\alpha_k \in \mathbb{R}$ and $P_k \in \mathbb{R}^{n\times s}$. 
Using these reformulations, we can implement the global methods and their smoothed variants for solving \eqref{AX-XC=B}. 
Thereafter, similar to \Cref{Theorem1}, we can evaluate the residual gap in finite precision arithmetic.

\begin{theorem}\label{Theorem6}
Let $X_k \in \mathbb{F}^{n\times s}$ and $R_k \in \mathbb{F}^{n\times s}$ be the $k$th approximation and residual, respectively, which are generated by \eqref{X_R_Sylv} in finite precision arithmetic. 
Then, the norm of the residual gap $G_{R_k} := (B - \mathcal A(X_k)) - R_k$ is bounded as follows: 
\begin{align}
\|G_{R_k}\| \leq k{\bf u}[(7+2m)\|A\| + (7+2s)\|C\|] \max_{0<j\leq k}\|X_j\| + 5(k+1){\bf u}\max_{0\leq j \leq k}\|R_j\|. \label{gap5} 
\end{align}
\end{theorem}
\begin{proof}
As in the proof of \Cref{Theorem1}, the local errors in the updated approximation and residual can be evaluated as follows: 
\begin{align*}
X_k 
&= X_{k-1} + \alpha_{k-1}P_{k-1} + E_{X_k},\quad \|E_{X_k}\| \leq {\bf u} (3\|X_{k-1}\| + 2\|X_k\|), \\
R_k 
&= R_{k-1} - \text{fl}(\alpha_{k-1}\text{fl}(\mathcal A(P_{k-1}))) - E_{R_k}',\quad \|E_{R_k}'\| \leq {\bf u} (2\|R_{k-1}\| + \|R_k\|) \\
&= R_{k-1} - \alpha_{k-1}\text{fl}(\mathcal A(P_{k-1})) - E_{R_k}' - E_{R_k}'',\quad \|E_{R_k}''\| \leq {\bf u} (\|R_{k-1}\| + \|R_k\|). 
\end{align*}
The result of $\alpha_{k-1}\text{fl}(\mathcal A(P_{k-1})) =\alpha_{k-1}\text{fl}( \text{fl}(AP_{k-1}) - \text{fl}(P_{k-1}C) )$ is expressed as 
\begin{align*}
&\alpha_{k-1}\text{fl}(\mathcal A(P_{k-1})) \\
&\quad= \alpha_{k-1}\text{fl}(AP_{k-1}) - \alpha_{k-1}\text{fl}(P_{k-1}C)-\alpha_{k-1}E_1,\\
&\phantom{\quad=}\quad \|\alpha_{k-1} E_1\| \leq {\bf u} |\alpha_{k-1}| (\|A\| + \|C\|)\|P_{k-1}\| \leq {\bf u} (\|A\| + \|C\|)(\|X_{k-1}\| + \|X_k\|)\\
&\quad= (\alpha_{k-1}AP_{k-1} + \alpha_{k-1}E_2) - (\alpha_{k-1}P_{k-1}C + \alpha_{k-1}E_3) -\alpha_{k-1}E_1,\\
&\phantom{\quad=}\quad \|\alpha_{k-1} E_2\| \leq m {\bf u} |\alpha_{k-1}| \|A\|\|P_{k-1}\|\leq m{\bf u} \|A\|(\|X_{k-1}\| + \|X_k\|),\\
&\phantom{\quad=}\quad \|\alpha_{k-1} E_3\| \leq s {\bf u} |\alpha_{k-1}| \|P_{k-1}\|\|C\|\leq s{\bf u} \|C\|(\|X_{k-1}\| + \|X_k\|)\\
&\quad= \alpha_{k-1}\mathcal A(P_{k-1}) + \alpha_{k-1}E_{R_k}''',\\
&\phantom{\quad=}\quad \|\alpha_{k-1}E_{R_k}'''\| \leq {\bf u}\left[ (1 + m)\|A\| + (1+s)\|C\| \right](\|X_{k-1}\| + \|X_k\|),
\end{align*}
where $E_{R_k}''' := E_2 - E_3 - E_1$. 
Thus, the local errors in $R_k$ are evaluated as follows: 
\begin{align*}
&R_k = R_{k-1} - \alpha_{k-1}\mathcal A(P_{k-1}) - E_{R_k},\\
&\quad \|E_{R_k}\| \leq {\bf u} (3\|R_{k-1}\| + 2\|R_k\|) + {\bf u}\left[ (1 + m)\|A\| + (1+s)\|C\| \right](\|X_{k-1}\| + \|X_k\|), 
\end{align*}
where $E_{R_k}:=E_{R_k}' + E_{R_k}'' + \alpha_{k-1}E_{R_k}'''$. 
Using the linearity of the operator $\mathcal A$, the norm of the residual gap can be bounded as follows: 
\begin{align*}
&\|(B-\mathcal A(X_{k})) - R_k\|\\
&\quad= \|B - \mathcal A(X_{k-1} + \alpha_{k-1}P_{k-1} + E_{X_k})- (R_{k-1} - \alpha_{k-1}\mathcal A(P_{k-1}) - E_{R_k})\|\\
&\quad= \|(B-\mathcal A(X_{k-1})) - R_{k-1} - (AE_{X_k} - E_{X_k}C) + E_{R_k}\|\\
&\quad \leq (\|A\| + \|C\|) \sum_{j=1}^k\|E_{X_j}\| + \sum_{j=1}^k\|E_{R_j}\|\\
&\quad\leq {\bf u}[(7+2m)\|A\| + (7+2s)\|C\|] \sum_{j=1}^k\|X_j\| + 5{\bf u}\sum_{j=0}^k\|R_j\|.
\end{align*}
Note that $(B-\mathcal A(X_0)) - R_0 = O$ holds for $X_0 = O$. 
The proof is completed by bounding $\sum_{j=1}^k\|X_j\|$ and $\sum_{j=0}^k\|R_j\|$ with $k\max_j \|X_j\|$ and $(k+1)\max_j \|R_j\|$, respectively.
\end{proof}

As the upper bound in \eqref{gap5} contains the term $\max_j \|R_j\|$ similarly to \eqref{gap1}, it is important to reduce the maximum of the residual norms to reduce the residual gap. 
When using CIRS, as $X_j$ and $R_j$ in \eqref{gap5} are replaced by $\hat Y_j$ and $\hat S_j$, respectively, the smoothed method is expected to have a smaller residual gap.

\section{Numerical experiments}\label{sec7}

We present numerical experiments that were conducted to demonstrate that a large residual gap occurs when a large relative residual norm exists in the global and block Lanczos-type solvers, which can be improved by CIRS.
We compared the convergence of Gl-CGS2, Gl-BiCGSTAB, Bl-BiCGSTABpQ, and Bl-BiCGGRrQ and their smoothed variants S-Gl-CGS2, S-Gl-BiCGSTAB, S-Bl-BiCGSTABpQ, and S-Bl-BiCGGRrQ (i.e., \Cref{alg2,alg3,alg4,alg5}) using several model problems.

Numerical calculations were carried out in double-precision floating-point arithmetic on a PC (Intel Core i7-8650U CPU with 16 GB of RAM) equipped with MATLAB R2018a. 
The iterations were started with $X_0 = O$ and were stopped when the relative norms of the recursively updated residuals ($\|R_k\|/\|B\|$ for the primary methods and $\|\hat S_k\|/\|B\|$ for the smoothed methods) were less than $10^{-14}$. 
We used several test matrices. 
One of these was a Toeplitz matrix $A=[a_{ij}] \in \mathbb{R}^{2000\times 2000}$ that was defined as $a_{ii} := 2$, $a_{i, i+1} := 1$, and $a_{i+4, i} := -1$ for each $i$ (otherwise, $a_{ij} := 0$). 
The other matrices were obtained from the SuiteSparse Matrix Collection \cite{Davis_2011}. 
\Cref{table_matrix} presents the dimension ($n$), number of nonzero entries (nnz), maximum number of nonzero entries per row ($m$), and the two-norm condition number ($\kappa_2(A)$). 
The Toeplitz matrix, cdde2, pde2961, and bfwa782 were used as the coefficient matrices of \eqref{AX=B}, and fs\_680\_1 and can\_24 were used for $A$ and $C$, respectively, in \eqref{AX-XC=B}. 
The right-hand side $B$ was provided as a random matrix and the initial shadow residual $R_0^{\bullet}$ was set to $R_0$ ($=B$) for all of the methods. 
Another shadow residual $R_0^{\circ}$ that was used in Gl-CGS2 and S-Gl-CGS2 was set to a random matrix. 
The conditions that were set for each experiment are described in the following.

The implementation of the compared methods naively followed the provided algorithms. In particular, we used the slash or backslash command in MATLAB to solve the small linear systems that appeared in the block methods. 
Moreover, the MATLAB command \textbf{qr}($\cdot$) was used to perform the QR factorization.

\begin{table}[t]
\centering
{\small
\caption{Characteristics of test matrices from SuiteSparse Matrix Collection \cite{Davis_2011}.}\label{table_matrix}
\begin{tabular}{llrrrr} \toprule
Target problem & Matrix & $n$ & nnz & $m$ & $\kappa_2(A)$ \\ \midrule
\multirow{3}{*}{Linear system \eqref{AX=B}} & cdde2 & 961 & 4,681 & 5 & 5.5e+01 \\
 & pde2961 & 2,961 & 14,585 & 5 & 6.4e+02 \\ 
 & bfwa782 & 782 & 7,514 & 24 & 1.7e+03 \\ \midrule 
\multirow{2}{*}{Sylvester equation \eqref{AX-XC=B}} & fs\_680\_1 & 680 & 2,184 & 8 & 1.5e+04 \\ 
 & can\_24 & 24 & 160 & 9 & 7.8e+01 \\ \bottomrule
\end{tabular}
}
\end{table}

\begin{figure}[t]
		\begin{minipage}{0.49\textwidth}
			\centering
			\includegraphics[scale=0.3]{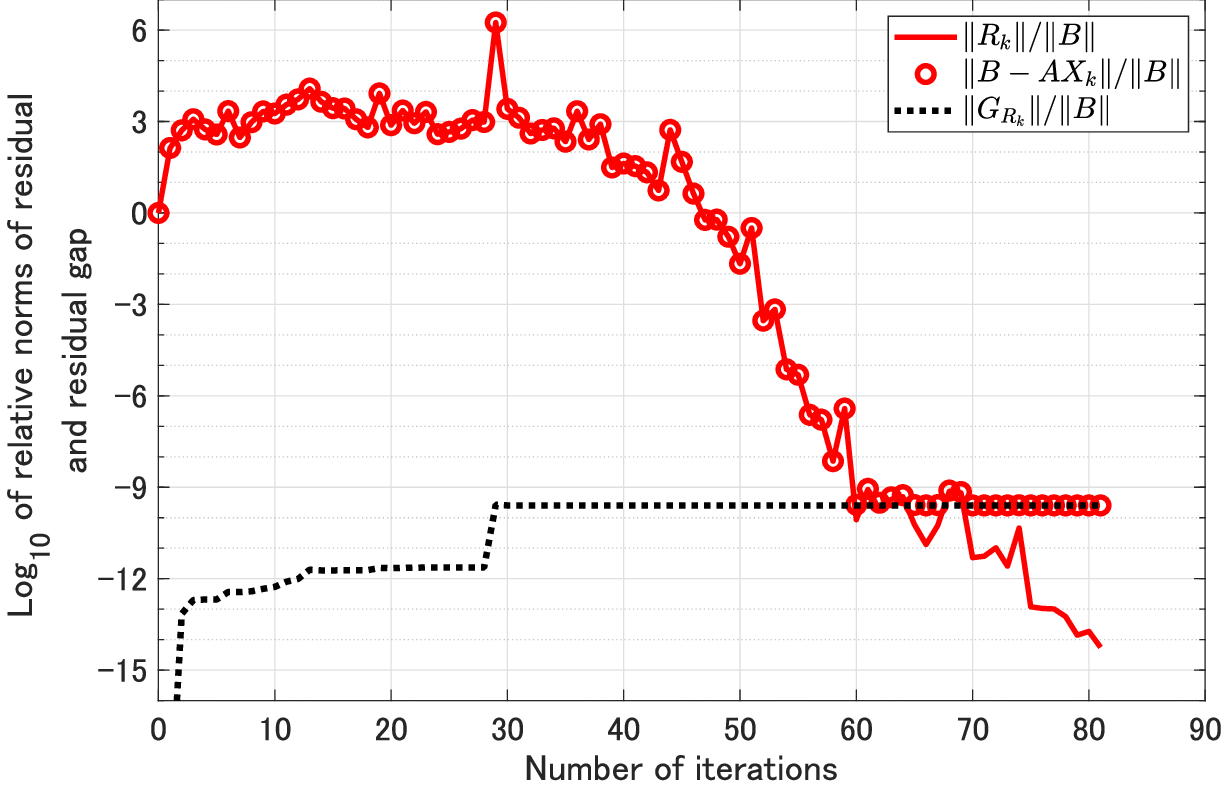}
		\end{minipage}
		\begin{minipage}{0.49\textwidth}
			\centering
			\includegraphics[scale=0.3]{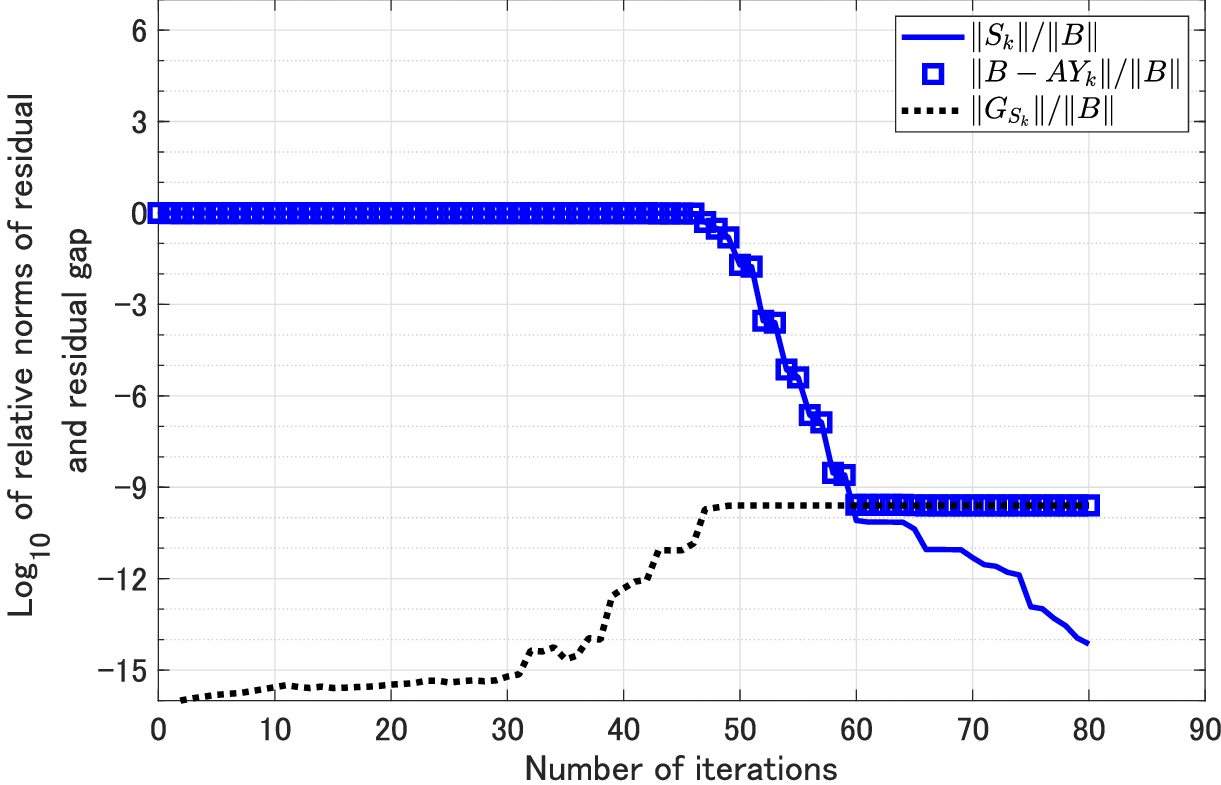}
		\end{minipage}
		\caption{Convergence histories of non-smoothed Gl-CGS2 (left) and smoothed Gl-CGS2 using SRS (right) for cdde2 with $s=16$.}\label{Ex1_1}
		\begin{minipage}{0.49\textwidth}
			\centering
			\includegraphics[scale=0.3]{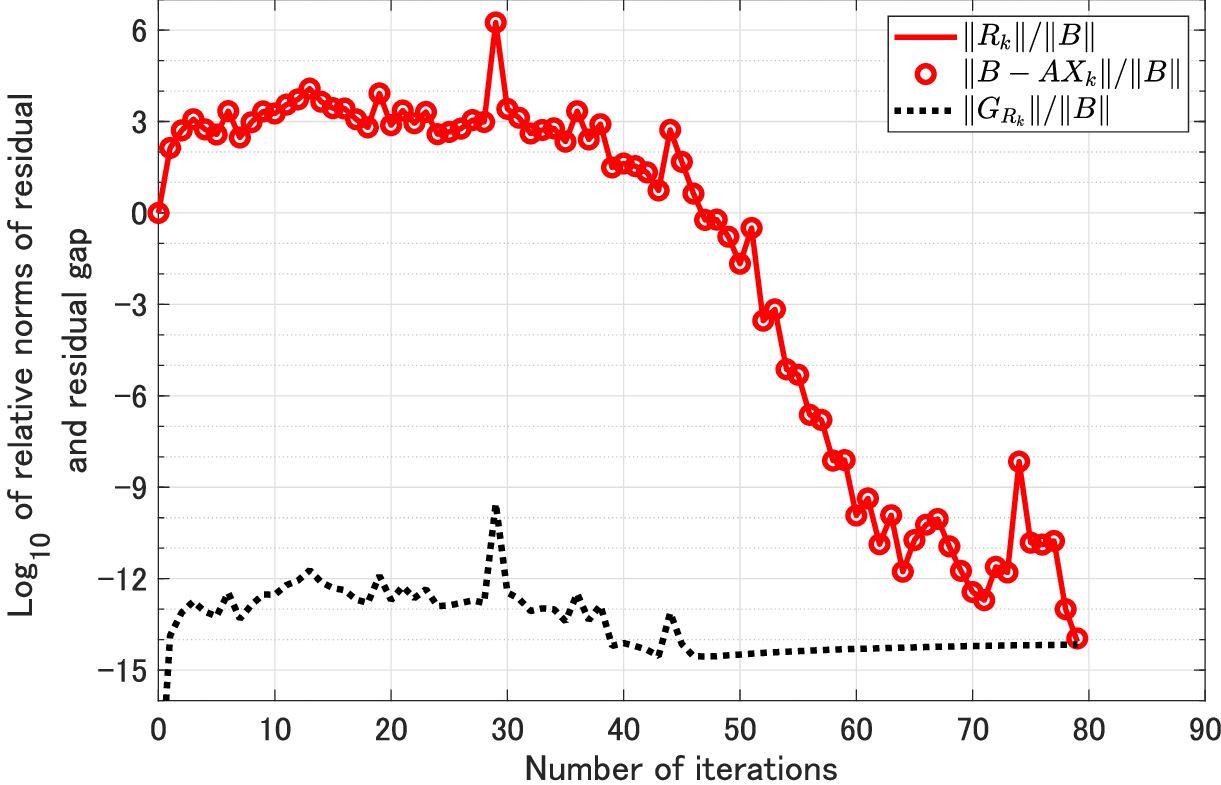}
		\end{minipage}
		\begin{minipage}{0.49\textwidth}
			\centering
			\includegraphics[scale=0.3]{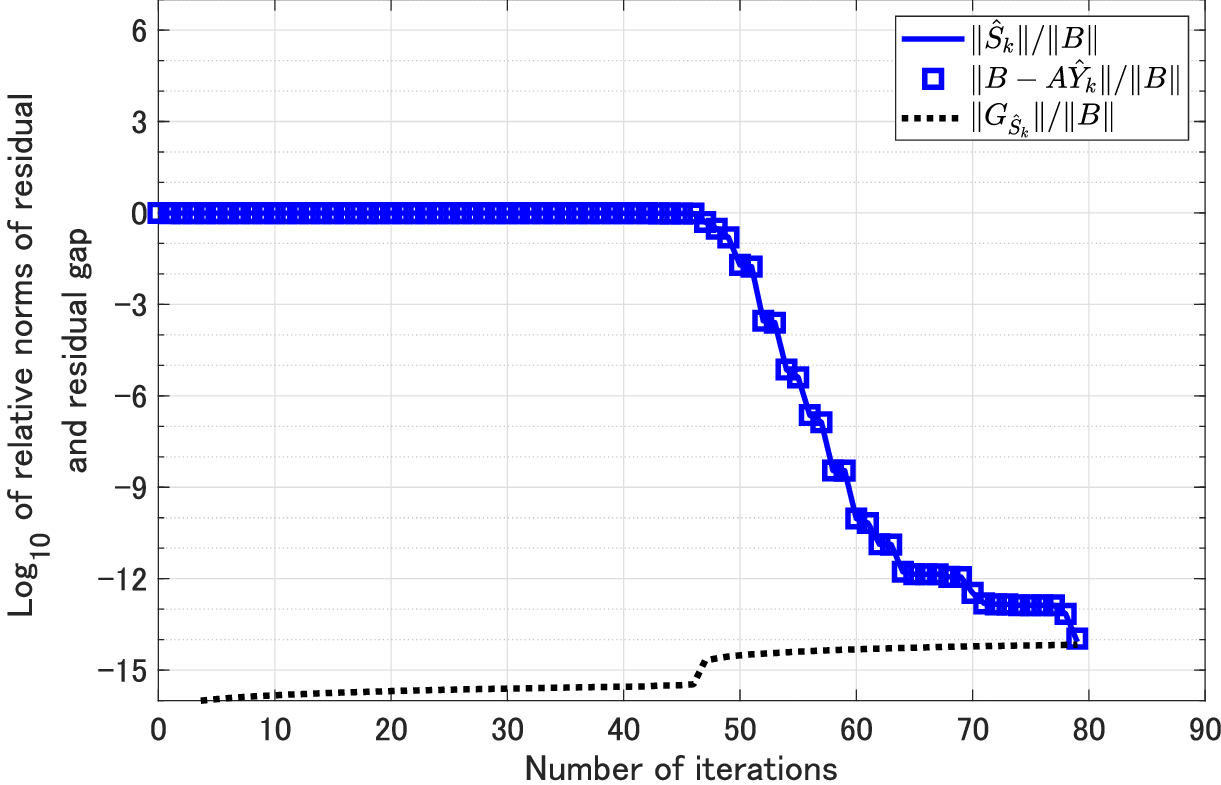}
		\end{minipage}
		\caption{Convergence histories of proposed S-Gl-CGS2 for cdde2 with $s=16$: primary sequences (left) and smoothed sequences (right).}\label{Ex1_2}
\end{figure}

\subsection{Differences between simple residual smoothing and CIRS}\label{sec7.1}

Following \cite{Aihara_2019}, we first present the advantages of CIRS compared to the simple residual smoothing scheme. 
We compared the convergence of Gl-CGS2 and its smoothed variants for \eqref{AX=B}, where $A$ was set to cdde2 and $s$ was set to 16. 
In the following, the simple residual smoothing \eqref{S-W} with \eqref{smo_eta} is referred to as SRS.

The left panel of \Cref{Ex1_1} depicts the histories of the relative norms of the recursively updated residuals $R_k$, true residuals $B-AX_k$, and residual gap $G_{R_k}$ of the non-smoothed Gl-CGS2. 
The plots indicate the number of iterations on the horizontal axis versus $\log_{10}$ of the relative norms on the vertical axis. 
The right panel of \Cref{Ex1_1} depicts the corresponding histories of the smoothed Gl-CGS2 using SRS. 
Here, the evaluated quantities are associated with $Y_k$ and $S_k$ in the smoothed sequences instead of $X_k$ and $R_k$ in the primary sequences, respectively. 
\Cref{Ex1_2} displays the histories for the proposed S-Gl-CGS2 using CIRS; the left and right panels present the primary and smoothed sequences obtained in lines~12 and 11, respectively, in \Cref{alg2}.

The following can be observed from \Cref{Ex1_1,Ex1_2}: 
Gl-CGS2 has a large relative residual norm, which results in a large residual gap, leading to a loss of attainable accuracy of the approximations, as indicated by \Cref{Theorem1}. 
In particular, $\|R_k\|$ increases drastically at the first and 29th iterations, and $\|G_{R_k}\|$ increases accordingly. 
The smoothed Gl-CGS2 using SRS exhibits smooth convergence behavior, but the residual gap is not improved; although $\|G_{S_k}\|$ is relatively small for $\|R_k\| \gg \|B\|$, it increases after 30 iterations and reaches the same order of magnitude as $\|G_{R_k}\|$ of Gl-CGS2 at the 47th iteration for $\|R_k\| \approx \|B\|$. 
This phenomenon follows \Cref{Theorem3}. 
In contrast, the behavior of $\|G_{R_k}\|$ in the proposed S-Gl-CGS2 is similar to that of $\|R_k\|$; that is, as indicated by \Cref{Theorem5}, $\|G_{R_k}\|$ increases drastically with a large increase in $\|R_k\|$, but decreases as $\|R_k\|$ becomes smaller. 
In the final iterations, $\|G_{R_k}\|$ of S-Gl-CGS2 is much smaller than that of Gl-CGS2. 
Furthermore, $\|\hat S_k\|$ of S-Gl-CGS2 converges smoothly with a small residual gap throughout the iterations, as indicated by \Cref{Theorem4}. 
It should be noted that the final sizes of $\|G_{R_k}\|$ and $\|G_{\hat S_k}\|$ in S-Gl-CGS2 are of the same order of magnitude.

\begin{figure}[t]
		\begin{minipage}{0.49\textwidth}
			\centering
			\includegraphics[scale=0.3]{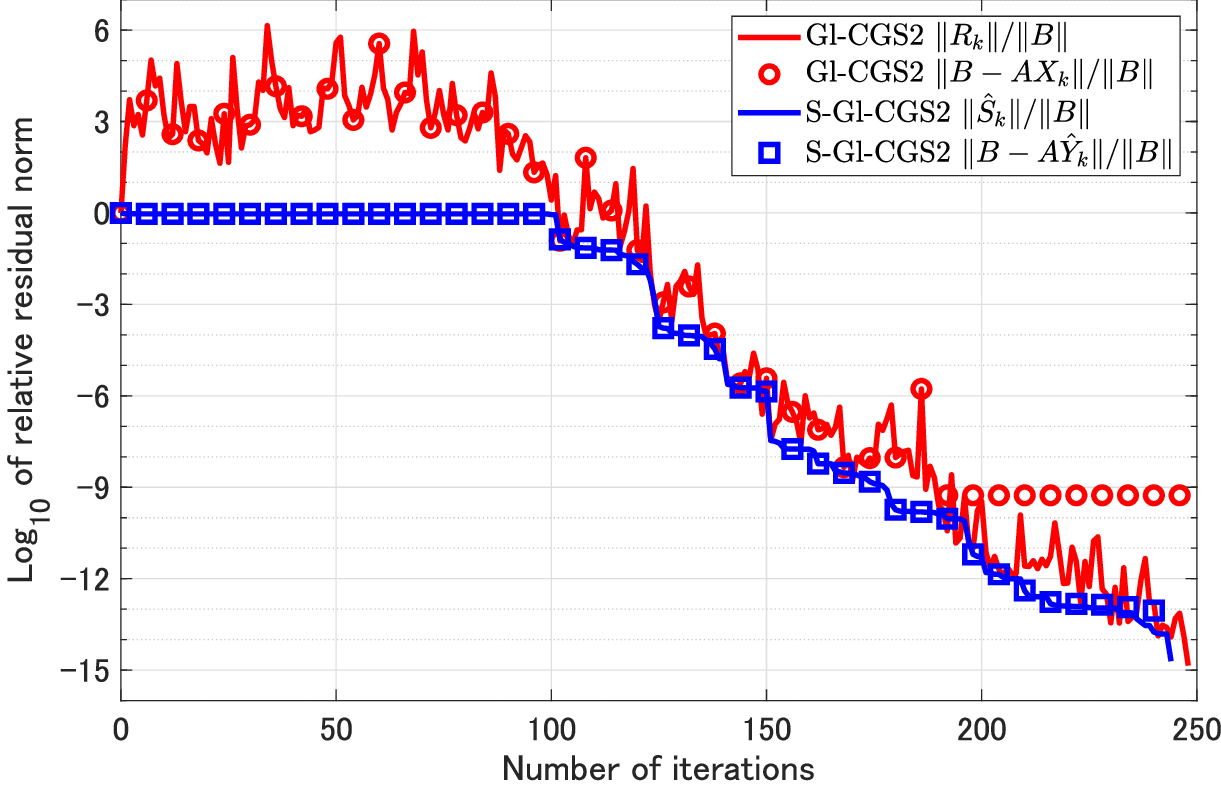}
		\end{minipage}
		\begin{minipage}{0.49\textwidth}
			\centering
			\includegraphics[scale=0.3]{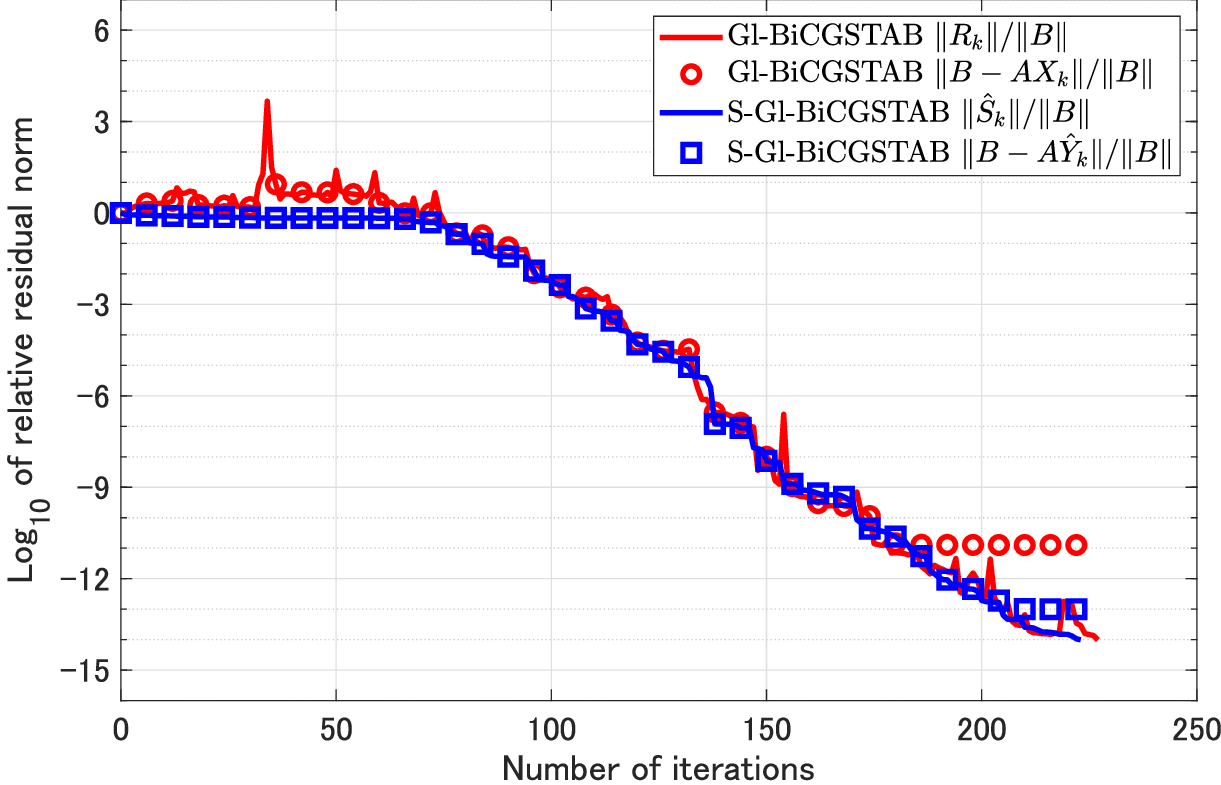}
		\end{minipage}
		\caption{Convergence histories of Gl-CGS2 and S-Gl-CGS2 (left), and Gl-BiCGSTAB and S-Gl-BiCGSTAB (right) for pde2961 with $s=16$.}\label{Ex2}
\end{figure}

\begin{table}[t]
\centering
{\small 
\caption{Number of iterations and true relative residual norm of Gl-CGS2, S-Gl-CGS2, Gl-BiCGSTAB, and S-Gl-BiCGSTAB for test matrices.}\label{Ex2_Table}
\begin{tabular}{llrrrrrr} \toprule
 & & \multicolumn{2}{c}{$s=8$} & \multicolumn{2}{c}{$s=16$} & \multicolumn{2}{c}{$s=32$} \\ \cmidrule{3-8}
Matrix & Solver & Iter. & True res. & Iter. & True res. & Iter. & True res. \\ \midrule
\multirow{4}{*}{Toeplitz}	&	Gl-CGS2	&	1373	&	3.5e$-$12	&	1257	&	4.0e$-$12	&	1183	&	5.9e$-$13	\\	
	&	S-Gl-CGS2	&	1155	&	1.9e$-$14	&	1097	&	2.0e$-$14	&	1120	&	1.9e$-$14	\\	
	&	Gl-BiCGSTAB	&	1277	&	1.1e$-$12	&	1252	&	3.9e$-$12	&	1243	&	5.9e$-$14	\\	
	&	S-Gl-BiCGSTAB	&	1306	&	2.2e$-$14	&	1289	&	2.2e$-$14	&	1304	&	2.1e$-$14	\\	\midrule
\multirow{4}{*}{cdde2}	&	Gl-CGS2	&	87	&	2.9e$-$11	&	81	&	2.5e$-$10	&	84	&	4.7e$-$11	\\	
	&	S-Gl-CGS2	&	82	&	9.0e$-$15	&	79	&	1.1e$-$14	&	81	&	1.2e$-$14	\\	
	&	Gl-BiCGSTAB	&	89	&	3.3e$-$14	&	96	&	3.8e$-$14	&	90	&	7.5e$-$14	\\	
	&	S-Gl-BiCGSTAB	&	88	&	1.1e$-$14	&	94	&	1.2e$-$14	&	90	&	1.3e$-$14	\\	\midrule
\multirow{4}{*}{pde2961}	& Gl-CGS2	&	221	&	1.6e$-$10	&	248	&	5.4e$-$10	&	228	&	5.6e$-$10	\\	
	&	S-Gl-CGS2	&	214	&	8.4e$-$14	&	244	&	8.8e$-$14	&	235	&	8.5e$-$14	\\	
	&	Gl-BiCGSTAB	&	220	&	2.9e$-$13	&	227	&	1.3e$-$11	&	225	&	2.1e$-$13	\\	
	&	S-Gl-BiCGSTAB	&	220	&	1.0e$-$13	&	223	&	1.0e$-$13	&	213	&	9.8e$-$14	\\	\midrule
\multirow{4}{*}{bfwa782} &	Gl-CGS2	&	356	&	3.5e$-$11	&	354	&	2.3e$-$11	&	356	&	3.7e$-$11	\\	
	&	S-Gl-CGS2	&	357	&	1.6e$-$13	&	354	&	1.7e$-$13	&	331	&	1.7e$-$13	\\	
	&	Gl-BiCGSTAB	&	$\dagger$	&	8.3e$-$13	&	$\dagger$	&	1.1e$-$12	&	1551	&	2.4e$-$12	\\	
	&	S-Gl-BiCGSTAB	&	$\dagger$	&	3.5e$-$13	&	1561	&	3.7e$-$13	&	$\dagger$	&	4.0e$-$13	\\	\bottomrule
\end{tabular}
}
\end{table}

\subsection{Experiments on smoothed global Lanczos-type solvers}\label{sec7.2}

We compared the convergences of Gl-CGS2, S-Gl-CGS2, Gl-BiCGSTAB, and S-Gl-BiCGSTAB for \eqref{AX=B} to demonstrate the effectiveness of the smoothed global methods, where $A$ was set to the Toeplitz matrix, cdde2, pde2961, and bfwa782 and $s$ was set to 8, 16, and 32. 
We set the maximum number of iterations to $2n$.

\Cref{Ex2} depicts the convergence histories of the relative norms of the recursively updated residuals and true residuals for pde2961 with $s=16$. 
The plots illustrate the number of iterations on the horizontal axis versus $\log_{10}$ of the relative residual norm on the vertical axis. 
For Gl-BiCGSTAB in the right panel, the maximum of the residual norms in the polynomial part was of the same order of magnitude as that in the BiCG part. 
For S-Gl-CGS2 and S-Gl-BiCGSTAB, we plotted the smoothed residual norms that were obtained from lines~11 and 9 of \Cref{alg2,alg3}, respectively. 
The true residual norms were plotted with markers using every 6th point.  
\Cref{Ex2_Table} displays the number of iterations required for successful convergence and the true relative residual norm at termination. 
The symbol $\dagger$ indicates that no convergence occurred within $2n$ iterations.

The following can be observed from \Cref{Ex2} and \Cref{Ex2_Table}: 
Gl-CGS2 often has a large residual gap and exhibits a loss of attainable accuracy of the approximations. 
Moreover, S-Gl-CGS2 has a relatively smaller residual gap and provides more accurate approximations in all cases. 
Although Gl-BiCGSTAB often has a smaller residual gap than Gl-CGS2, the residual gaps of S-Gl-BiCGSTAB and S-Gl-CGS2 are not larger than those of Gl-BiCGSTAB. 
The convergence speed of the smoothed methods is not significantly different from that of their non-smoothed counterparts. 
Gl-BiCGSTAB and S-Gl-BiCGSTAB fail to converge for bfwa782. 
For this problem, Gl-CGS2 and S-Gl-CGS2 are more robust, and S-Gl-CGS2 has a smaller residual gap.

\subsection{Experiments on Sylvester equation}\label{sec7.3}

We applied Gl-CGS2, S-Gl-CGS2, Gl-BiCGSTAB, and S-Gl-BiCGSTAB to the Sylvester equation \eqref{AX-XC=B}. 
Following \cite{Zhang_2010}, we set the matrices $A$ and $C$ in \eqref{AX-XC=B} to fs\_680\_1 and can\_24, respectively. 
\Cref{Ex3_Table} displays the number of iterations required for successful convergence and the true relative residual norm at termination.

\begin{table}[t]
\centering
{\small 
\caption{Number of iterations and true relative residual norm of Gl-CGS2, S-Gl-CGS2, Gl-BiCGSTAB, and S-Gl-BiCGSTAB for Sylvester equation.}\label{Ex3_Table}
\begin{tabular}{lrr} \toprule
Solver & Iter. & True res. \\ \midrule
Gl-CGS2	&	466	& 	5.5e$-$11	\\	
S-Gl-CGS2	&	484	& 	1.1e$-$14	\\	
Gl-BiCGSTAB	&	1113	 & 	6.0e$-$14	\\	
S-Gl-BiCGSTAB	&	1133	&	1.8e$-$14	\\	\bottomrule
\end{tabular}
}
\end{table}

It can be observed from \Cref{Ex3_Table} that although Gl-BiCGSTAB and S-Gl-BiCGSTAB have a small residual gap, many iterations are required. 
However, Gl-CGS2 converges much faster than Gl-BiCGSTAB, as observed in \cite{Zhang_2010}. 
Moreover, Gl-CGS2 has a large residual gap and exhibits a loss of attainable accuracy of the approximations. 
S-Gl-CGS2 converges as rapidly as Gl-CGS2 and has a smaller residual gap, and thus, it provides more accurate approximations.

\subsection{Experiments on smoothed block Lanczos-type solvers}\label{sec7.4}

We compared the convergences of Bl-BiCGSTABpQ, S-Bl-BiCGSTABpQ, Bl-BiCGGRrQ, and S-Bl-BiCGGRrQ for \eqref{AX=B} to demonstrate the effectiveness of the smoothed block methods. 
We refer the reader to \cref{sec7.2} for the computational conditions.

\Cref{Ex4} displays the convergence histories for cdde2 with $s=32$. 
\Cref{Ex4_Table} presents the number of iterations and the true relative residual norm at termination. 
For the plots of the figures and the notations in the table, we refer the reader to \cref{sec7.2}.

The following can be observed from \Cref{Ex4} and \Cref{Ex4_Table}: 
Bl-BiCGSTABpQ and Bl-BiCGGRrQ converge in all cases and their convergence speeds are comparable. 
However, there are cases in which both non-smoothed methods have a large relative residual norm as well as a large residual gap. 
Furthermore, S-Bl-BiCGSTABpQ and S-Bl-BiCGGRrQ exhibit smooth convergence behavior, a reduced residual gap, and more accurate approximations than their non-smoothed counterparts in most cases. 
We note that S-Bl-BiCGGRrQ sometimes decelerates the convergence speed for a larger $s$; in particular, it does not converge for bfwa782 with $s = 16$ and $32$. 
In the final subsection, we consider the cause of this phenomenon and present a strategy for its improvement.

\begin{figure}[t]
		\begin{minipage}{0.49\textwidth}
			\centering
			\includegraphics[scale=0.3]{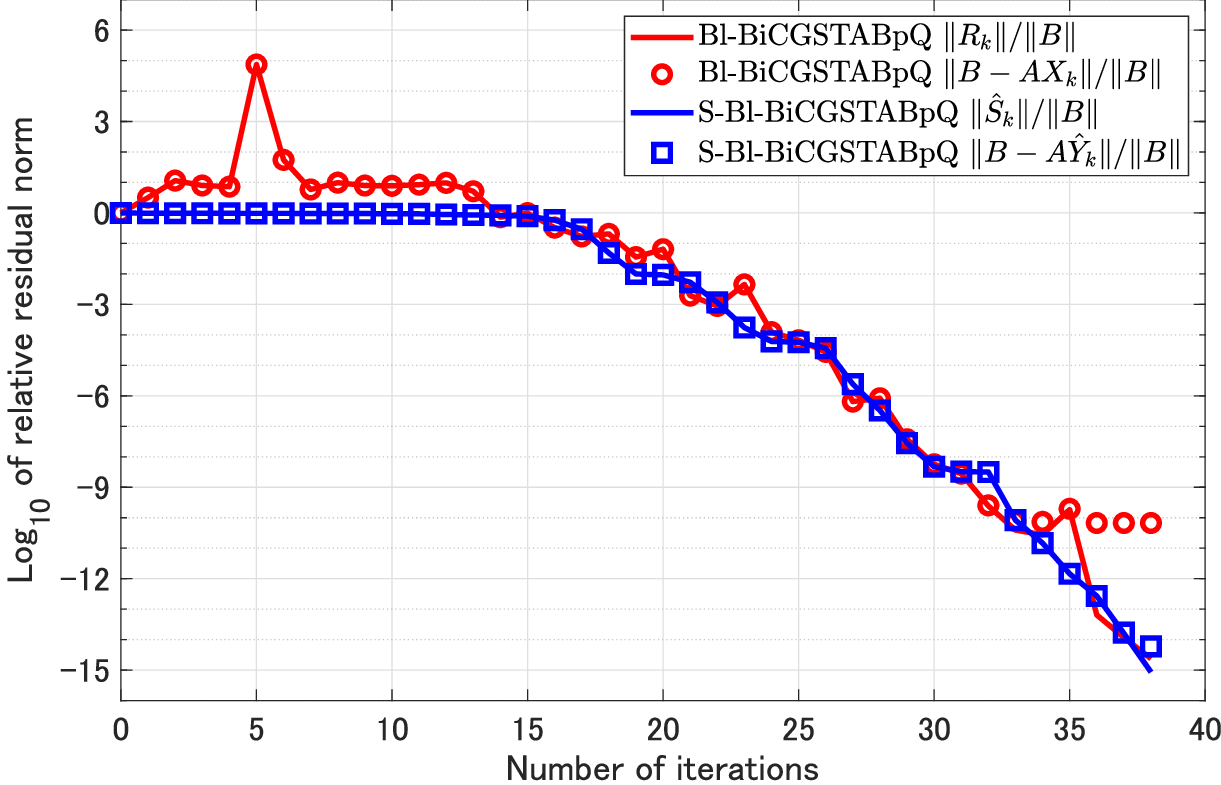}
		\end{minipage}
		\begin{minipage}{0.49\textwidth}
			\centering
			\includegraphics[scale=0.3]{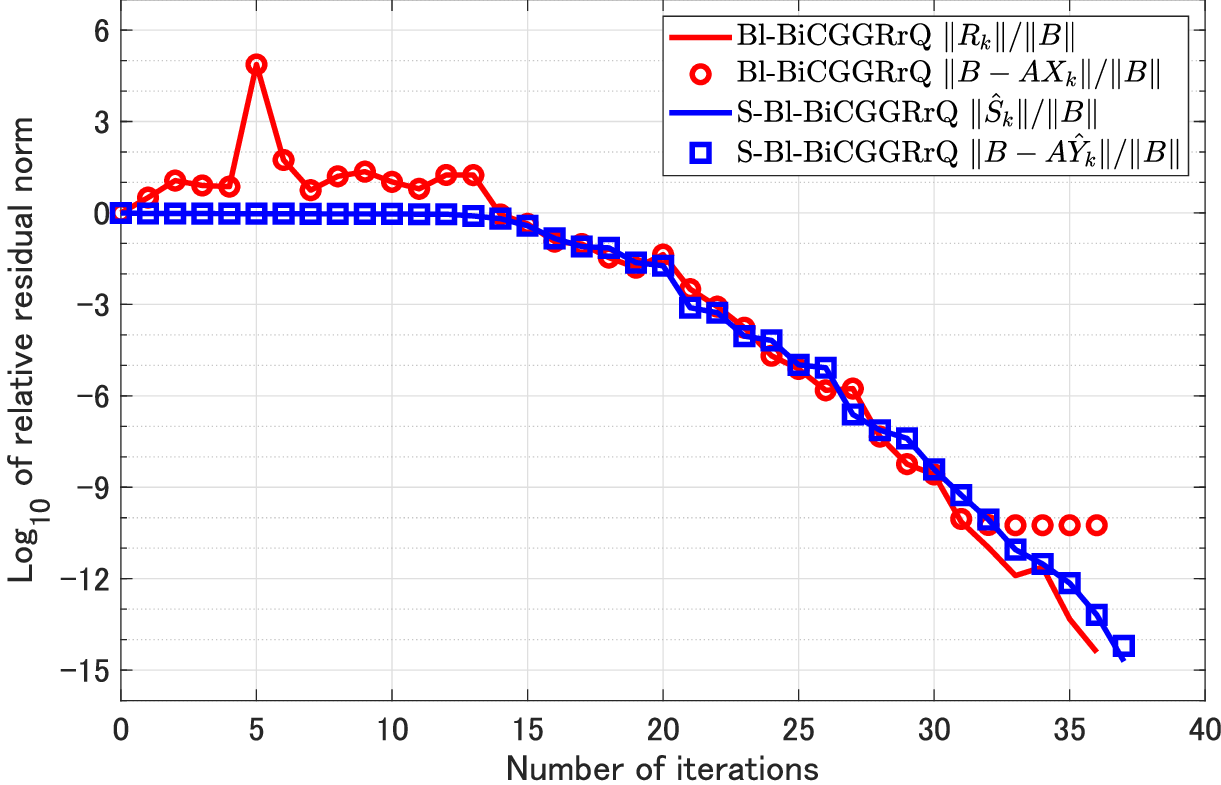}
		\end{minipage}
		\caption{Convergence histories of Bl-BiCGSTABpQ and S-Bl-BiCGSTABpQ (left), and Bl-BiCGGRrQ and S-Bl-BiCGGRrQ (right) for cdde2 with $s=32$.}\label{Ex4}
\end{figure}

\begin{table}[t]
\centering
{\small 
\caption{Number of iterations and true relative residual norm of Bl-BiCGSTABpQ, S-Bl-BiCGSTABpQ, Bl-BiCGGRrQ, and S-Bl-BiCGGRrQ for test matrices.}\label{Ex4_Table}
\begin{tabular}{llrrrrrr} \toprule
 & & \multicolumn{2}{c}{$s=8$} & \multicolumn{2}{c}{$s=16$} & \multicolumn{2}{c}{$s=32$} \\ \cmidrule{3-8}
Matrix & Solver & Iter. & True res. & Iter. & True res. & Iter. & True res. \\ \midrule
\multirow{4}{*}{Toeplitz}	&	Bl-BiCGSTABpQ	&	253	&	1.4e$-$13	&	123	&	1.4e$-$13	&	69	&	3.0e$-$13	\\	
	&	S-Bl-BiCGSTABpQ	&	224	&	1.1e$-$14	&	123	&	9.1e$-$15	&	67	&	5.4e$-$15	\\	
	&	Bl-BiCGGRrQ	&	227	&	9.3e$-$14	&	132	&	1.2e$-$13	&	70	&	4.8e$-$14	\\	
	&	S-Bl-BiCGGRrQ	&	220	&	1.3e$-$14	&	121	&	1.0e$-$14	&	68	&	6.7e$-$15	\\	\midrule
\multirow{4}{*}{cdde2}	&	Bl-BiCGSTABpQ	&	65	&	2.9e$-$14	&	55	&	1.6e$-$13	&	38	&	6.7e$-$11	\\	
	&	S-Bl-BiCGSTABpQ	&	64	&	1.0e$-$14	&	54	&	8.0e$-$15	&	38	&	6.0e$-$15	\\	
	&	Bl-BiCGGRrQ	&	63	&	1.3e$-$12	&	53	&	1.3e$-$13	&	36	&	5.7e$-$11	\\	
	&	S-Bl-BiCGGRrQ	&	63	&	1.0e$-$14	&	52	&	1.1e$-$14	&	37	&	6.3e$-$15	\\	\midrule
\multirow{4}{*}{pde2961}	&	Bl-BiCGSTABpQ	&	113	&	3.7e$-$13	&	86	&	5.0e$-$12	&	65	&	5.0e$-$13	\\	
	&	S-Bl-BiCGSTABpQ	&	120	&	7.4e$-$14	&	87	&	6.1e$-$14	&	62	&	5.1e$-$14	\\	
	&	Bl-BiCGGRrQ	&	136	&	2.3e$-$13	&	126	&	3.7e$-$13	&	74	&	1.1e$-$12	\\	
	&	S-Bl-BiCGGRrQ	&	123	&	7.6e$-$14	&	120	&	7.4e$-$14	&	92	&	6.5e$-$14	\\	\midrule
\multirow{4}{*}{bfwa782}	&	Bl-BiCGSTABpQ	&	80	&	2.9e$-$12	&	57	&	1.2e$-$12	&	37	&	1.6e$-$12	\\	
	&	S-Bl-BiCGSTABpQ	&	86	&	8.5e$-$14	&	66	&	7.8e$-$14	&	43	&	6.6e$-$14	\\	
	&	Bl-BiCGGRrQ	&	97	&	2.5e$-$12	&	63	&	9.8e$-$13	&	34	&	1.2e$-$12	\\	
	&	S-Bl-BiCGGRrQ	&	88	&	8.7e$-$14	&	$\dagger$	&	4.6e$-$10	&	$\dagger$	&	3.0e$-$13	\\	\bottomrule
\end{tabular}
}
\end{table}

\subsection{Strategy of partial application of CIRS}

In S-Bl-BiCGGRrQ, the system $T_{k+1} \xi_k = R_{k+1}$ needs to be solved for $T_{k+1}$ at each iteration, where $\xi_k$ is the R-factor of the QR factorization of $R_k$. 
This may cause numerical instability when the condition number of $\xi_k$ becomes large, leading to a loss of convergence speed, as can be observed in \Cref{Ex4_Table}. 
This difficulty can probably be remedied by the partial application of CIRS. 
Specifically, we perform CIRS only when the condition $\|\hat S_k\|/\|B\| > \theta$ holds for a threshold value $\theta$; otherwise, we use the standard updating process. 
This simple strategy enables us to avoid the inversion of an ill-conditioned matrix $\xi_k$, and the convergence speed is expected to be maintained. 
Moreover, the residual gap is expected to be reduced even in the partially smoothed iteration process, because the essence of CIRS is to suppress the maximum of the residual norms. 
This strategy can also be applied to S-Bl-BiCGSTABpQ.

We conducted numerical experiments to demonstrate the effectiveness of the smoothed methods using the above switching strategy. 
The threshold value $\theta$ was set to $10^{-2}$ and we refer the reader to \cref{sec7.4} for the other computational conditions.

\Cref{Ex5} displays the convergence histories of the relative norms of the recursively updated residuals of Bl-BiCGGRrQ and S-Bl-BiCGGRrQ with and without the switching strategy for bfwa782 with $s=16$. 
The plots indicate the number of iterations on the horizontal axis versus $\log_{10}$ of the relative residual norm on the left vertical axis. 
In \Cref{Ex5}, we also display the history of the condition number of $\xi_k$, where $\log_{10}$ of the condition number is plotted on the right vertical axis. 
\Cref{Ex5_Table} depicts the number of iterations and the true relative residual norm at termination for S-Bl-BiCGSTABpQ and S-Bl-BiCGGRrQ with the switching strategy.

It can be observed from \Cref{Ex5} that S-Bl-BiCGGRrQ with the switching strategy converges at the same speed as Bl-BiCGGRrQ, whereas S-Bl-BiCGGRrQ without the strategy decelerates the convergence speed with an increase in the condition number of $\xi_k$. 
By comparing \Cref{Ex4_Table,Ex5_Table}, it can be observed that the switching strategy is useful for maintaining the convergence speed while improving the attainable accuracy. 
We note that the selection of the threshold value $\theta$ does not cause a severe problem. 
In our experience, similar effects can be observed for $\theta = 10^{-1}, 10^{-2}$, and $10^{-3}$, 
for example.

\begin{figure}[t]
\centering
\includegraphics[scale=0.3]{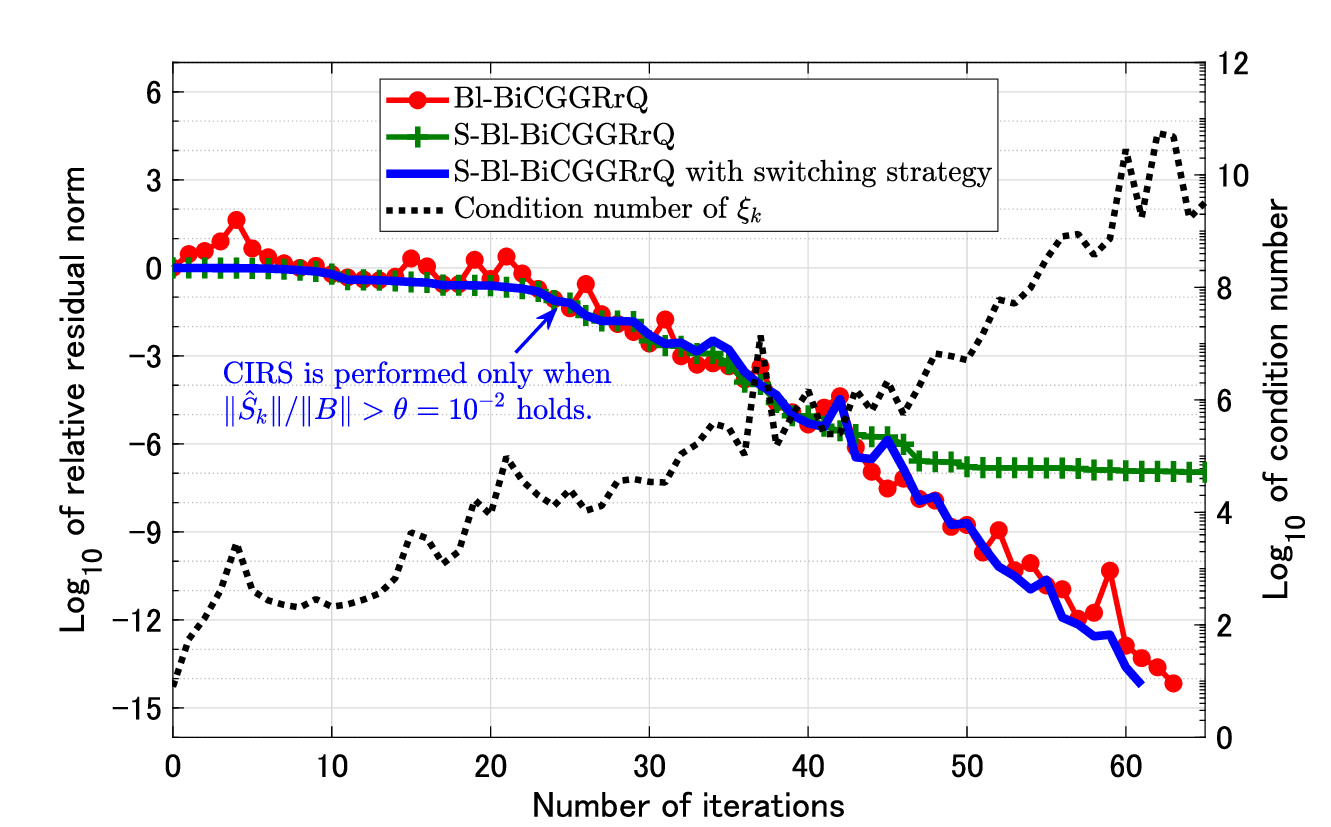}
\caption{Convergence histories of Bl-BiCGGRrQ and S-Bl-BiCGGRrQ with and without switching strategy, and history of the condition number of $\xi_k$ for bfwa782 with $s=16$.}\label{Ex5}
\end{figure}

\begin{table}[t]
\centering
{\small
\caption{Number of iterations and true relative residual norm of S-Bl-BiCGSTABpQ and S-Bl-BiCGGRrQ with switching strategy for test matrices.}\label{Ex5_Table}
\begin{tabular}{llrrrrrr} \toprule
 & & \multicolumn{2}{c}{$s=8$} & \multicolumn{2}{c}{$s=16$} & \multicolumn{2}{c}{$s=32$} \\ \cmidrule{3-8}
Matrix & Solver & Iter. & True res. & Iter. & True res. & Iter. & True res. \\ \midrule
\multirow{2}{*}{Toeplitz}	&	S-Bl-BiCGSTABpQ	&	228	&	1.2e$-$14	&	119	&	1.1e$-$14	&	67	&	5.3e$-$15	\\	
	&	S-Bl-BiCGGRrQ	&	239	&	1.1e$-$14	&	113	&	1.0e$-$14	&	64	&	1.7e$-$14	\\	\midrule
\multirow{2}{*}{cdde2}	&	S-Bl-BiCGSTABpQ	&	64	&	1.1e$-$14	&	53	&	9.5e$-$15	&	37	&	1.1e$-$14	\\	
	&	S-Bl-BiCGGRrQ	&	63	&	8.4e$-$15	&	52	&	7.7e$-$15	&	36	&	1.1e$-$14	\\	\midrule
\multirow{2}{*}{pde2961}	&	S-Bl-BiCGSTABpQ	&	122	&	9.4e$-$14	&	86	&	7.7e$-$14	&	61	&	6.5e$-$14	\\	
	&	S-Bl-BiCGGRrQ	&	121	&	7.5e$-$14	&	107	&	7.0e$-$14	&	73	&	5.6e$-$14	\\	\midrule
\multirow{2}{*}{bfwa782}	&	S-Bl-BiCGSTABpQ	&	84	&	1.1e$-$13	&	64	&	9.7e$-$14	&	43	&	8.4e$-$14	\\	
	&	S-Bl-BiCGGRrQ	&	90	&	9.2e$-$14	&	61	&	7.5e$-$14	&	37	&	6.0e$-$14	\\	\bottomrule
\end{tabular}
}
\end{table}

\section{Concluding remarks}\label{sec8}

We have presented a rounding error analysis to show that global and block Lanczos-type solvers may have a large residual gap when a large relative residual norm exists during the iterations. 
To reduce the residual gap, we extended cross-interactive residual smoothing for a single linear system to the case of multiple right-hand sides and designed several smoothed algorithms for these solvers. 
The proposed algorithms can be implemented with few additional costs compared to their non-smoothed counterparts.
The numerical experiments demonstrated that the smoothed variants have a smaller residual gap and provide more accurate approximations than their original counterparts.

In this study, we restricted the rounding error analysis to the case of real numbers. 
However, our results are also valid for the complex case, because the error bounds for basic complex arithmetic have similar forms to those of real arithmetic (cf.~\cite[section~3.6]{Higham_2002}). 
A detailed discussion on this point will be provided in future studies.

\end{document}